\documentclass[11pt,a4paper]{article}

\def\spacingset#1{\renewcommand{\baselinestretch}%
{#1}\small\normalsize} \spacingset{1.2}

\usepackage{caption,subcaption}
\usepackage{amssymb}
\usepackage{graphicx}
\usepackage{amsmath}
\usepackage{latexsym}
\usepackage{bm}
\usepackage{amsfonts}
\usepackage{latexsym}
\usepackage{amssymb}
\usepackage{amsmath}
\usepackage{bm}
\usepackage{graphicx}
\usepackage{verbatim}
\usepackage[english]{babel}
\usepackage{latexsym, textcomp, alltt}
 \usepackage{mathptmx}

\def\nbI{\ensuremath{\mathrm{1\!l}}}
\newtheorem{theorem}{Theorem}
\newtheorem{lemma}{Lemma}
\newtheorem{corollary}{Corollary}
\newtheorem{proposition}{Proposition}

\newtheorem{remark}{Remark\/}

\newenvironment{proof}[1][Proof]{\textbf{#1.} }{\ \rule{0.5em}{0.5em}}

\begin{document}

\title{On detecting {\it weak} changes in the mean of CHARN models}

\author{{\bf Joseph Ngatchou-Wandji}$^{a}$\footnote{Corresponding author} \ and {\bf Marwa Ltaifa}$^b$\\
\\
 $^a$EHESP Sorbonne Paris Cit\'e  and 
\\
 Institut \'Elie Cartan de Lorraine, 54506 Vandoeuvre-l\`es-Nancy cedex, France.\\ E-mail: joseph.ngatchou-wandji@univ-lorraine.fr\\
\\
$^b$Institut \'Elie Cartan de Lorraine, 54506 Vandoeuvre-l\`es-Nancy cedex, France and
\\ LAMMDA-ESST Hammam-Sousse, University of Sousse, 4011 Hammam-Sousse, Tunisia.
 \\  E-mail: ltaifa.marwa@gmail.com}

\date{}
\maketitle

\noindent
{\bf Abstract}
We study a likelihood ratio test for detecting multiple {\it weak} changes in the mean of a class of  CHARN models.
The locally asymptotically normal (LAN) structure of the family of likelihoods under study is established. It results that the test is asymptotically optimal, and an explicit form of its asymptotic local power is given as a function of candidates change locations and changes magnitudes. Strategies for weak change-points detection and their locations estimates are described. The estimates are obtained as the time indices maximizing an estimate of the local power. A simulation study shows the good performance of our methods compared to some existing approaches. Our results are also applied to three sets of real data.
\vspace{.5cm}

\noindent
{\bf Running headline.} Change detection in CHARN models
\vspace{.5cm}

\vspace{.5cm}

\vspace{.25cm}

\section{Introduction}
This paper deals with {\it weak} changes detection in the mean of  Conditional Heteroscedastic AutoRegressive Nonlinear (CHARN) models (see \cite{hardle_ty}). We mean by {\it weak} change that whose magnitude is too small. Such a change can be produced by a phenomenon which changes gradually around a given date.
A weak change may be a precursor sign of upcoming more significant changes indicating a critical behavior.
 This can happen in economics and finance, public health, biosciences, engineering, climatology, hydrology, linguistics, genomics and in many other domains.
Because weak changes may be confused with noise effects, because they may be missing by usual change detection methods, it can be of primary importance to find {\it screening} methods more adapted to their detection.

Since \cite{page1}, change-point study has attracted a lot of attention, and has been the subject of a large amount of work. The monographs of \cite{bas} and \cite{csor} present the basic notions and theory on the subject. For long, the studies have been done in the context of independent observations as in 
 \cite{cher}, \cite{kan}, \cite{gard}, \cite{mac}, \cite{mat}, \cite{hac},
\cite{yao}, \cite{sen}, \cite{pet} and \cite{wolfe}. These last decades there is a growing attention on change-points study in dependent data. In this context, 
\cite{laj2} propose several methods of on-line change detection in the variance of an heteroskedastic time series. \cite{vogel1} constructs Wald-type tests for breaks detection in the trend of a dynamic time series. \cite{vogel2} and \cite{laj2} study tests for detecting change in the mean of various time series models. 
\cite{bardet_kw}, \cite{bardet_k},  \cite{kengne}, 
 \cite{berk}, \cite{gom0}, \cite{gom} and \cite{mohr} study tests for change detection in the parameters of various sub-classes of nonlinear or nonlinear heteroscedastic time series models.
%
%
%
\cite{ciuperca} gives a general criterion for estimating the number of breaks in a multidimensional nonlinear model. She establishes the consistency of the estimator obtained. 
%
%
\cite{doring1} proposes an estimation of change-points by using $U$-statistics, while   \cite{doring2} studies the convergence of such estimators. Additional change-point studies based on $U$-statistics can be found in \cite{laj}, \cite{deh13}, \cite{deh17} and \cite{nga}.
Other papers on change-point analysis are, among others, \cite{amano},  \cite{enikeeva_MW}, \cite{fotopoulos_JT}, \cite{huh}, \cite{zhou}, \cite{yang}, \cite{wang},  \cite{lav}, \cite{haw} and \cite{men}. 
\cite{aue} and \cite{tru} give surveys of recent methods and techniques for change-point study in time series.

However, except perhaps \cite{fotopoulos_JT}, \cite{khakhubia} and \cite{bat2} where weak changes are considered in the context of independent data, most of the existing studies generally focus on the detection of abrupt changes that may be seen with naked eyes. 
Moreover, for testing hypothesis approaches, the theoretical local power is rarely studied.  For instance, such a study seems to have been done recently in \cite{nga} and earlier in  \cite{deh17},  which studies 
 the local power of the CUSUM and Wilcoxon tests in the context of long range dependence shifted stationary time series. So, to the best of our knowledge, the study of weak changes in time series remains very little explored.

There are many techniques for multiple change-points detection in the literature. A popular one is the so-called binary segmentation (BS). The first step of this method is to search for a change-point in the entire data set. Once it is found the initial set of data is split into two subsets relatively to the change-point found. A similar search is done on each subset. The procedure continues until a given criterion is fulfilled. The wild binary segmentation (WBS) introduced in \cite{fryz} attempts to eliminate some of the weaknesses of BS by using at its first step, a random subset of the data at hand, rather than the entire data. 
%
%
The Narrowest-Over-Treshold (NOT) change-point detection proposed by \cite{bara} is derived in the context of nonparametric function estimation. The number and locations of the possible features owned by the function are first estimated before the function itself is parametrically estimated between the pairs of  adjacent detected features. Another detection method is the Isolate-Detect (ID) derived in \cite{anas}. It is based on an isolation strategy, which prevents the consideration of intervals that contain more than one change-point. This isolation allows for detection in the presence of frequent changes of possibly small magnitudes.  Other multiple change-points detection materials are given in \cite{niu}.

The methods proposed here are for detecting weak changes in off-line (or retrospective) data, and for estimating their number and locations. They are in spirit, models choice approaches, based on the theoretical power of a likelihood-ratio test for discriminating between contiguous time-varying mean CHARN models built on the basis of a segmentation of the observations. 
The theoretical power of the test, expressed as a function of the candidate break-points and their magnitudes, plays an important role in the detection of the changes and the estimation of their number and locations. 
The starting point of our methods is to find possible a priori information on the data studied. These can be obtained from their time series plot (chronogram), which may exhibit features as potential stationary areas and/or potential change-points locations. 
Assume from the chronogram that the first $m$ data are stationary. Our technique for change detection consists in segmenting the data into two subsets, one of which containing data with indices $1, \ldots, t$ for $t=m+1, m+2 \ldots, n$, where $n$ is the sample size. For each $t$, if the power of the test is close to its nominal level, then no change-point is found in the original data. Otherwise, there is at least one change-point to be found in the vicinity of the potential change-points suspected from the time series plot. The change locations are then estimated by the vector of indices maximizing the estimator of the local power obtained by substituting in its expression, estimators of the magnitudes of the breaks for the unknown magnitudes. The dimension of this vector is an estimator of the number of changes. 

From the chronogram 
one can get features of the data that can help searching for their change-points where they are more likely to be. 
This way, our methods attempt to seek changes only in some restricted areas of the data, while most of the usual methods would make an exhaustive search, which is well known to be time consuming and which may fail in finding weak changes. 
The use of the chronogram may be not necessary when some a priori information are available. An example is that of a piece-wise stationary data produced by a phenomenon known to change gradually around a given date $t_0$. If one wishes to detect the date of the beginning of the change, the probability of finding it in the vicinity of $t_0$ is larger than that of finding it elsewhere. Many practical such examples can be found among data associated with important events that happened at known dates, and for which one wishes to know when the changes induced by  these events started.

In our methods, instead of maximizing CUSUM-type statistics for estimating the number of changes and their locations, we rather maximize the local power of our test. 
The techniques for computing this important tool are more close to those of \cite{huang1} and \cite{huang2} or \cite{benghabrit1} and \cite{benghabrit2}, than to those in \cite{bas} and \cite{des}. 
We use them in Section 3, where we present our test as well as the associated theoretical results based on the models, the notation and assumptions listed in Section 2. 
In Section 4, a simulation experiment conducted shows a good performance of our methods compared to WBS, NOT, ID and 
the CUSUM and {\it weighting} CUSUM tests studied in \cite{laj3}. All these methods, as well as ours, are also applied to real data from the floods of the Upper Hanjiang River in China. Section 5 contains the proofs of the theoretical results stated in Section 3.
%
%
%
%
\section{Testing some contiguous CHARN models}
Our methods rest on the local power of a likelihood-ratio test for contiguous time-dependent coefficients CHARN models. Before studying the test, we precise the notation, we present the models and list the main assumptions.
\subsection{The notation}
Let $d, d_1,d_2$ be positive integers with $d=d_1+d_2$ and $ \mathcal{G} $ be a real-valued function defined and differentiable on $ \mathbb R^d $. We denote for $z=(z_1, \ldots, z_d) =(z^{(1)}, z^{(2)})\in \mathbb R^d $ and for any $k=1,2$
$$ \partial_z {\mathcal{G}(z)}
: = \left (\dfrac {\partial {\mathcal{G} (z)}} {\partial z_1}, \ldots, \dfrac {\partial {\mathcal{G}(z)}} {\partial z_d} \right), \ \ 
\partial_{z^{(k)}}{\mathcal{G}(z)}
: = \left (\dfrac {\partial {\mathcal{G} (z)}} {\partial z_1^{(k)}}, \ldots, \dfrac {\partial {\mathcal{G}(z)}} {\partial z_{d_k}^{(k)}} \right)
$$
$$ \partial_z^2 {\mathcal{G} (z)}
: = \left (\dfrac {\partial^2 {\mathcal{G} (z)}} {\partial {z_i } \partial {z_j}}: 1 \leq i, j \leq d \right), \partial_{z^{(k)}}^2 {\mathcal{G} (z)}
: = \left (\dfrac {\partial^2 {\mathcal{G} (z)}} {\partial z_i^{(k)} \partial z_j^{(k)}}: 1 \leq i, j \leq d_k \right)$$
$$\partial_{z^{(1)} z^{(2)}}^2 {\mathcal{G} (z)}
: = \left (\dfrac {\partial ^2 {\mathcal{G} (z)}} {\partial {z_i^{(1)}} \partial {z_j^{(2)}}}: \ 1 \leq i \leq d_1, \ \ 1 \leq j \leq d_2 \right).$$
Let $ U = (U_1, \ldots, U_d)^{\top}$ and $D = (D_1^{\top}, \ldots, D_d^{\top})^{\top},$ where for all $ i = 1, \ldots, d, $ $ D_i $ is a matrix $ d \times 1 $ and $ H^{\top} $ stands for the transpose of a vector or a matrix $ H $. 
For any number $c$, we recall that $c U=U c=(c U_1, \ldots, c U_d)^{\top}$ and we define $ D U^{\top} = U^{\top} D = (D_1^{\top}U_1, \ldots D_d^{\top}U_d)^{\top} $. We denote by $ || \cdot ||$ the Euclidean norm on $\mathbb R^d$ and by $||A||_M=\max_{1\leq i \leq n}{\sum_{j=1}^n{|a_{ij}|}},$ the norm of any square matrix $A=(a_{ij})_{1\leq i,j\leq n}.$
%
For a differentiable function $h$ with derivative $h'$, we denote 
$$ \phi_h ={h' \over h} \ \ \mbox{  and  } I(h)=\int_{-\infty}^\infty\phi_h^2(x)h(x)dx.$$
%
\subsection{The models and the assumptions}

Let $k,n \in \mathbb N$ and $k<< n$. For the full description of our methodologies, we assume the observations $X_1,\ldots,X_n$ are issued from 
the following time-varying coefficients $p$th-order CHARN($p$) model (see e.g., \cite{hardle_ty}). 
\begin{equation} \label{mod}
 X_t=T(Z_{t-1})+\gamma^{\top}\omega(t)+V(Z_{t-1})\varepsilon_t,\quad t\in\mathbb Z,
 \end{equation}
where 
$T$ and $V$ are real-valued functions with $\inf_{x\in\mathbb R^p} V(x)>0$, 
$\gamma$$=$$(\gamma_1,\ldots,\gamma_k,\gamma_{k+1})^{\top}$$\in\mathbb{R}^{k+1}$, 
$\omega(t)=(\nbI_{[t_0,t_1)}(t),\nbI_{t[t_1,t_2)}(t),\ldots,\nbI_{[t_{k-1},t_k)}(t),\nbI_{[t_k,t_{k+1})}(t))^{\top}$ with 
%
$1=t_0<t_1<\ldots<t_k< t_{k+1}=n$ standing for 
the candidates locations of the changes, 
 $(X_t)_{t\in\mathbb Z}$ is piece-wise stationary and ergodic on the $[t_j,t_{j+1})$'s, $(\varepsilon_t)_{t\in\mathbb Z}$ is a white noise with density function $f$, and  for all $t\in\mathbb Z,$  $Z_t=(X_t,\ldots,X_{t-p+1})^{\top}$ and 
$\varepsilon_t$ is independent of the $ \sigma $-algebra $G_{t-1}=\sigma(Z_1,\ldots,Z_{t-1})$. 

The class of models (\ref{mod}) contains common models such as AR($p$), ARCH($p$), TARCH($p$), EXPAR($p$) whose statistical and probability properties are widely studied in the literature (see, eg, \cite{tong}). Such a model, of infinite order, is considered in \cite{bardet_k}. 
%

In the next section, we construct a likelihood-ratio test for testing
 $$H_0:\gamma=\gamma_0\ \ \ \text{ against }\ \ \ H^{(n)}_{\beta}:\gamma=\gamma_0+\dfrac{\beta}{\sqrt{n}}=\gamma_{n},\quad n>1,$$
for some $\gamma_0 \in \mathbb R^{k+1}$ and $\beta \in \mathbb R^{k+1}$ depending on the $t_j$'s. 
We study its properties under $H_0$ and $H^{(n)}_{\beta}$. 

For any $k\ge 0$, denote by $\mathcal P_{k,t^k}$ the theoretical power of this test at $t^k=(t_1,\ldots,t_k)$, with the convention that $\mathcal P_{0,t^0}=\alpha$, $\alpha \in (0,1)$, the level of the test. An expression of $\mathcal P_{k,t^k}$ is given in the next section as a function of the parameters. This is computed under the following general assumptions: 
\begin{itemize}
\item[$(A_1):$] $f$ is differentiable and $\displaystyle\int_{-\infty}^\infty xf(x)dx=0$ and $\displaystyle\int_{-\infty}^\infty{x^2 f(x)} dx=1.$
\item[$(A_2):$] $\phi_f$ is differentiable with a $c_\phi$-Lipschitzian derivative $\phi'_f$, where $0<c_\phi<\infty.$
\item[$(A_3):$] $\displaystyle \lim_{x\rightarrow {+\infty}}{f(x)}=\displaystyle \lim_{x\rightarrow {-\infty}}{f(x)}=\displaystyle\lim_{x\rightarrow {+\infty}}{f'(x)}= \displaystyle\lim_{x\rightarrow {-\infty}}{f'(x)}=0.$
\item[$(A_4): $] $\displaystyle\int_{-\infty}^\infty|\phi_f(x)|^3 f(x) dx <\infty.$
\item[$(A_5) : $] For all $j=1,\ldots,k+1,$ $n_j(n)=t_j-t_{j-1}\longrightarrow\infty$, $\displaystyle {n_j(n) \over n}\longrightarrow\alpha_j$  as $n\rightarrow\infty$.
\item[$(A_6) : $]  The sequence of the $Z_t$'s whose components are associated with indices within $[t_{j-1}, t_j)$, $j=1,\ldots,k+1,$ is
 stationary and ergodic with stationary cumulative distribution function $F_j$.
%
\item[$(A_7):$] $\mu_{j\ell}(\gamma_0)=I(f){\displaystyle\int_{{\mathbb R}^p}{\dfrac{1}{V^\ell (x)}}dF_j(x)}<\infty$, \  
$j=1,\ldots,k+1,$ $\ell \leq3.$
\end{itemize}
\begin{remark} \label{rm1} 
\begin{itemize}
\item[(i)] In the case $\gamma_0 \ne 0$, assumption $(A_6)$ may not hold. But one can check that it does at least for the subclass of (\ref{mod}) with constant functions $T$. These models have the form of those studied in \cite{deh13} and \cite{deh17}.
\item[(ii)] On each $[t_{j-1}, t_j)$, $j=1,\ldots,k+1,$ there are at most $p$ random vectors 
$Z_t$ whose components are associated with indices within both $[t_{j-1}, t_j)$ and $[t_{j-2}, t_{j-1})$. They may not have the same stationary distribution. But since their number is negligible behind $n_j(n)$, their distributions do not affect the asymptotic results.
\end{itemize}
\end{remark}
\begin{remark} \label{rm2}
\begin{itemize}
\item[(i)] The $\mu_{j\ell}(\gamma_0)$'s 
are functions of $\gamma_0$ through the $F_j$'s.
\item[(ii)]
$(A_3)$ implies $\displaystyle\int_{-\infty}^\infty{\phi_f(x)f(x)}dx=0$ and $\displaystyle\int_{-\infty}^\infty{\phi^2_f(x)f(x)}dx=\displaystyle\int_{-\infty}^\infty{\phi'_f(x)f(x)}dx. $
\item[(iii)] 
We assume that as $n$ tends to infinity, the density function of $(\varepsilon_0, Z_0)^{\top}$ under $H_{\beta}^{(n)}$ tends to its density under $H_0$. With this, asymptotically, $(\varepsilon_0, Z_0)^{\top}$ has no influence on the test statistic derived here. As it allows for a simplification of the form of the likelihoods, we consider in this paper the likelihoods and log-likelihoods  associated with $(\varepsilon_0, Z_0^{\top}, X_1,X_2, \ldots, X_n)^{\top}$.
%
\end{itemize}
\end{remark}
\begin{remark}
A decision about $H_0$ and  $H^{(n)}_{\beta}$ can be taken on the basis of the power $\mathcal P_{k,t^k}$. Since this quantity is unknown, an estimator $\widehat{\mathcal P}_{k,t^k}$ can be obtained by substituting the estimators of the parameters for the parameters in its expression. In particular, 
for $j=2, \ldots k+1$, 
the $j$th  component of $\beta$ can be estimated by $\sqrt{n}(\overline X_{j,n}-\widehat{\widetilde{\gamma}}_{0j,n})$, where 
$\widehat{\widetilde{\gamma}}_{0j,n}$ is an estimator of the $j$th component of $\gamma_0$, and $\overline X_{j,n}$ is the sample mean of the observations with time indices within $[t_j,t_{j+1})$. A possible $\widehat{\widetilde{\gamma}}_{0j,n}$ is the sample mean of the observations with indices in $[t_{j-1}, t_j)$. 
\end{remark}
\subsection{Changes detection and their locations estimation}
Many strategies for changes detection and for estimating their locations in off-line data have been proposed in the literature. Some are briefly reviewed in the introduction. Some others a reviewed in \cite{niu1} and in \cite{niu}. The strategies that we propose here rest on the power of the test studied in this paper. Before presenting them, 
we first discuss the forms of $\gamma_0$ and $\beta$ in (\ref{mod}). 

One would invoke a change in the mean of the data when at a certain time indice  the mean changes. Thus, the reference mean would be that of the earliest observations assumed to be stationary up to the time index $t_1$ where the first change occurs or is assumed to occur. 
Hence, the first component of $\gamma_0$ will always be taken to be the theoretical mean $\mu_1$ of the data on $[t_0, t_1)$, while the first component of $\beta$ would be nil. In 
\cite{bardet_k} it is explained how to check the stationarity of 
these earliest observations called there {\it historical data}. But the method may not work for such historical data containing weak changes.

\cite{and} 
 proposes change-point tests assuming the location of the potential change is known {\it a priori} and the number of observations before or after the change-point presumed fixed. 
\cite{bus} gives a survey of change detection methods with known or unknown potential change-points. All these suggest that in change-point study  one may need some {\it a priori} information about the data that may help avoiding an exhaustive search of the change-point locations.

It is well known (see, eg, \cite{bro}, page 14) that one of the starting points of the modelization of a time series $X_1, X_2, \ldots, X_n$ is the analysis of its chronogram. This can provide  some information as for example, a guessed  
historical data $X_1, X_2, \ldots, X_m$, $m<<n$, a maximum number $K$ of breaks and areas where they may be located, hence, a minimum distance $h<<n$ between adjacent breaks. 

Recall $ \mathcal P_{0,t^0}=\alpha$. Let $\zeta \in (0,.1)$ be a given positive small number. Our procedures for detecting changes and estimating their locations in a time series $X_1, X_2, \ldots, X_n$ starts with finding the above a priori information and continues as follows.
\begin{itemize}
\item[S1 -] {\bf Change detection:} Take $k=1$ and apply the likelihood-ratio test derived here to model (\ref{mod}) for the $t_1$'s satisfying  $m\le t_1 \le n$. 
\begin{itemize}
\item[i-] If $|\widehat{\mathcal P}_{1,t^1}- \mathcal P_{0,t^0}| \le \zeta$ for all such $t_1$,  then, no change is detected in the series.
\item[ii-] If $|\widehat{\mathcal P}_{1,t^1}- \mathcal P_{0,t^0}|>\zeta$ for a $t_1$, then, there is at least one change in the series. 
\end{itemize}
\item[S2 -]  {\bf Estimating change-point locations:} Let $m<\tau_1^0<\ldots< \tau_K^0\le n-h$, $\tau_j^0-\tau_{j-1}^0 \ge h, \ j=2,\ldots,K$,  be the  potential locations of the changes obtained from the chronogram. Let $\mathcal C_j$ be an arbitrary set of time indices containing $\tau_j^0$, and $\mathcal C_j \cap \mathcal C_\ell= \emptyset$, $j\ne\ell=1, \ldots,K$. 
For any $\ell=1, \ldots,K$, consider 
 $\mathcal S_\ell=\bigcup_{1 \le j_1< j_2< \ldots< j_\ell \le K}\mathcal C_{j_1} \times \mathcal C_{j_2} \times \ldots \times \mathcal C_{j_\ell}$.
For any $\ell$-tuple $\tau^\ell=(\tau_1,\ldots, \tau_\ell) \in \mathcal S_\ell$, apply the testing problem in Subsection 2.2 with  $t_j=\tau_j$, $j=1, \ldots,\ell$ and compute $\widehat{\mathcal P}_{\ell,t^\ell}$. Assume that 
$\mathcal C_1$ contains a change-point location.
%
%
%
\medskip

\noindent
$\bullet$ For $\ell=2, \ldots, K$, 
\begin{itemize}
\item[i-] If $|\max_{t^\ell \in \mathcal S_\ell} \widehat{\mathcal P}_{\ell,t^\ell}-\max_{t^{\ell-1} \in \mathcal S_{\ell-1}} \widehat{\mathcal P}_{\ell-1,t^{\ell-1}}|\le \zeta$, there are probably $\ell-1$ change-point locations in the areas studied. Then, put $\widetilde{\mathcal S}_\ell=\mathcal S_{\ell-1}$.

If $|\max_{t^\ell \in \mathcal S_\ell} \widehat{\mathcal P}_{\ell,t^\ell}-\max_{t^{\ell-1} \in \mathcal S_{\ell-1}} \widehat{\mathcal P}_{\ell-1,t^{\ell-1}}|> \zeta$, there are probably $\ell$ change-point locations in the areas studied. Then, put $\widetilde{\mathcal S}_\ell=\mathcal S_\ell$.
\smallskip

\noindent
Denote by $\iota$ the length of the vectors in $\widetilde{\mathcal S}_K$. Then an estimate $(\widehat k, \widehat t^{\widehat k}) $ of the couple of the number of changes and the vector of the locations can be obtained as 
$$(\widehat k, \widehat t^{\widehat k}) =\arg \max_{t^\iota \in \widetilde{\mathcal S}_K} \widehat{\mathcal P}_{\iota,t^\iota}.$$
\end{itemize}
%
%
%
%
%
%

$\bullet$ Clearly, if $k$, the number of changes is known, $t^k$ can be estimated by 
\begin{equation} \label{sin}
\widehat t^k= \arg \max_{t^k \in \mathcal S_k} \widehat{\mathcal P}_{k,t^k}.
\end{equation}
\item[S3 -] {\bf Alternative to} S2 : 
As a substitute to S2, we can use the following sequential strategy. 
Consider the segments 
$$\mho_1 =\{1,\ldots, m, \ldots, \tau_1^0+h \}, \ \  
\mho_\ell =\{\tau_{\ell-1}^0+h, \ldots, \tau_\ell^0+h\}, \ \ \ell=2,\ldots,K.$$
\begin{itemize}
\item[i- ] Step 1 : 
Apply the testing problem in Subsection 2.2 for $k=1$ and for each $t_1 \in \mho_1$, $t_1 >m$, with $\tau_1^0+h$ substituted for $n$.

If $|\widehat{\mathcal P}_{1,t^1}- \mathcal P_{0,t^0}|\le \zeta$ for all these $t_1$'s, there is no change location on $\mho_1$. Then update $\mho_2$ by merging the current $\mho_2$ with $\mho_1$.
\smallskip

\noindent
If $|\widehat{\mathcal P}_{1,t^1}- \mathcal P_{0,t^0}|> \zeta$ for a $t_1$, there is one change location on $\mho_1$ estimated by 
$$\widehat t_1^1= \arg \max_{t^1 \in \mho_1} \widehat{\mathcal P}_{1,t^1}.$$
Then update $\mho_2$ by substituting $\widehat t_1^1$ for $\tau_1^0$ in the current $\mho_2$.
\item[ii- ] Step $\ell$ : Apply the testing problem in Subsection 2.2 for $k=1$ and for each $t_1 \in \mho_\ell$, $t_1 >\tau_\ell^0+h$ or $t_1 >\widehat t_\ell^1+h$ , with the length of $\mho_\ell$ substituted for $n$.

If $|\widehat{\mathcal P}_{\ell,t^\ell}- \mathcal P_{0,t^0}|\le \zeta$ for all these $t_1$'s, there is no change location on $\mho_\ell$. Then update $\mho_{\ell+1}$ by merging the current $\mho_{\ell+1}$ with $\mho_\ell$.
\smallskip

\noindent
If $|\widehat{\mathcal P}_{\ell,t^\ell}- \mathcal P_{0,t^0}|> \zeta$ for some $t_1$, there is one change location on $\mho_\ell$ estimated by 
$$\widehat t_\ell^1= \arg \max_{t^1 \in \mho_\ell} \widehat{\mathcal P}_{\ell,t^1}.$$
Then update $\mho_{\ell+1}$ by substituting $\widehat t_\ell^1$ for $\tau_\ell^0$ in the current $\mho_{\ell+1}$.

\end{itemize}
It is clear that the total number of changes in the series is the number of the $\mho_\ell$'s where a change has been located, and the locations are estimated by the $ \widehat t_\ell^1$'s.
\end{itemize}
\medskip

\noindent
The part i of S1 comes from the idea that in the case of no change, all the estimates of the possible $\beta$ would be close to 0, and the estimated power would be close to the level of the test. Part ii is motivated by the fact that if there is at least one change, then one can find one $t_1$ for which the mean of the observations before and after that time index are different. The test would then reject the null hypothesis of no break. S2 and S3 are built from the same ideas.
%
%
\begin{remark}
\begin{itemize}
\item[i- ] In practice, the sets $\mathcal C_j$'s in the strategy S2 must not be too large for larger $K$.
\item[ii- ] In the strategy S1, for the $t_1$'s approaching $n$, the data can be spliced as in the so-called circular binary segmentation method described in \cite{ols}. Then, in the case of one single change, for estimating its location, one can directly apply (\ref{sin}) with $\mathcal S_1=\{m+1, \ldots, n\}$
 or S3 with ${\mho}_1=\{m+1, \ldots, n\}$, splicing the data when the current time index approaches $n$.
\end{itemize}
\end{remark}

\section{The theoretical results}
\subsection{A LAN result}
\noindent 
We first establish the contiguity of the sequences $\{H_0^{(n)}= H_0 \}$ and 
$\{H_0^{(n)}= H^{(n)}_{\beta} \}$, $\beta \in \mathbb R^{k+1}$. We denote by $\Lambda_n(\gamma_0,\beta)$ the log-likelihood ratio of $ H_0 $ against $ H^{(n)}_{\beta}$ 
and 
we define the {\it central statistic} 
\begin{equation} \label{cent}
\Delta_n(\gamma_0,\beta)=\dfrac{1}{\sqrt{n}}\sum_{t=1}^n{\dfrac{\beta^{\top}\omega(t)}{V(Z_{t-1})}\phi_f[\varepsilon_t(\gamma_0)]},
\end{equation}
where for all $\gamma=(\gamma_1,\ldots,\gamma_{k+1})\in\mathbb R^{k+1}$ and all $t\in\mathbb Z$, 
\begin{equation} \label{epsi}
\varepsilon_t(\gamma)=\dfrac{X_t-T(Z_{t-1})-\gamma^{\top}\omega(t)}{V(Z_{t-1})}.
\end{equation}
%
%
\begin{theorem} \label{th1}
Assume that $ (A_1) $-$ (A_7) $ hold. Then, for any $\beta \in \mathbb R^{k+1}$, under $ H_0, $ as $n\rightarrow\infty,$ 
$$\Lambda_n(\gamma_0,\beta)=\Delta_n(\gamma_0,\beta)-\dfrac{\mu(\gamma_0,\beta)}{2}+o_P(1),$$
$$\Delta_n(\gamma_0,\beta)\overset{D}{\longrightarrow} \mathcal{N}\left(0,\mu(\gamma_0,\beta)\right),$$ 
with $$\mu(\gamma_0,\beta)=\sum_{j=1}^{k+1}{\alpha_j\beta^2_j\mu_{j2}}(\gamma_0) = \varpi^2(\gamma_0,\beta).$$ 
\end{theorem}
\begin{proof}
See Appendix.
\end{proof}

\begin{corollary} \label{cr1}
Assume that $(A_1)$-$(A_7)$ hold. Then, for any $\beta \in \mathbb R^{k+1}$ the sequences
 $\lbrace H^{(n)}_{\beta}:n\geq 1\rbrace $ and  $\lbrace H^{(n)}_0=H_0:n\geq 1\rbrace $ are contiguous. Moreover, under $H^{(n)}_{\beta}$, as $n\rightarrow\infty,$
$$\Delta_n(\gamma_0,\beta)\overset{D}{\longrightarrow} \mathcal{N}(\mu(\gamma_0,\beta),\mu(\gamma_0,\beta)).$$
\end{corollary}
\begin{proof}
For any $\beta \in \mathbb R^{k+1}$, from Theorem \ref{th1}, under $H_0,$ as $n\rightarrow\infty,$
$$\Delta_n(\gamma_0,\beta)\overset{D}{\longrightarrow} \mathcal{N}(0,\mu(\gamma_0,\beta)).$$
Therefore, under $H_0$, as $n\rightarrow\infty,$
$$\Lambda_n(\gamma_0,\beta)\overset{D}{\longrightarrow} \mathcal{N}\left(-\dfrac{\mu(\gamma_0,\beta)}{2},\mu(\gamma_0,\beta)\right).$$
Then, it is easy to see that under $H_0,$ as $n\rightarrow\infty,$
$$\begin{pmatrix} 
\Delta_n(\gamma_0,\beta) \\ 
\Lambda_n(\gamma_0,\beta)
 \end{pmatrix}\overset{D}{\longrightarrow} \mathcal{N}\left(
 \begin{pmatrix} 
0 \\ 
-\dfrac{\mu(\gamma_0,\beta)}{2}
 \end{pmatrix} , 
 \begin{pmatrix} 
\mu(\gamma_0,\beta) &\mu(\gamma_0,\beta) \\ 
\mu(\gamma_0,\beta) & \mu(\gamma_0,\beta)
 \end{pmatrix} \right).$$
It results from \cite{lecam} or \cite{dreosebeke} that $\lbrace H^{(n)}_{\beta}:n\geq 1\rbrace $ and $\lbrace H^{(n)}_0=H_0:n\geq 1\rbrace $ are contiguous and under $H^{(n)}_{\beta},$ as $n\rightarrow\infty,$ 
$$\Delta_n(\gamma_0,\beta)\overset{D}{\longrightarrow}\mathcal{N}\left( \mu(\gamma_0,\beta),\mu(\gamma_0,\beta)  \right).$$
\end{proof}

\noindent 
For known $\gamma_0$ and for any $\beta \in \mathbb R^{k+1}$, for testing $H_0$ against $H^{(n)}_{\beta},$ we base our test on the statistic 
$$ \mathcal{T}_n(\gamma_0,\beta)= \dfrac{\Delta_n(\gamma_0,\beta)}{\widehat{\varpi}_n(\gamma_0,\beta)},$$
where $\widehat{\varpi}_n(\gamma_0,\beta)$ is any consistent estimator of $\varpi(\gamma_0,\beta)= \mu^{\frac{1}{2}}(\gamma_0,\beta)$. In the sequel, $\widehat{\varpi}_n(\gamma_0,\beta)$ will be taken to be the natural estimator  
$\widehat{\mu}_n^{\frac{1}{2}}(\gamma_0,\beta)$ with $\widehat{\mu}_n(\gamma_0,\beta)=\sum_{j=1}^{k+1}{\widehat{\alpha}_j\beta^2_j\widehat{\mu}_{j2}}(\gamma_0),$ and for $ j=1,\ldots,k+1$, 
 $\widehat{\alpha}_j$ is an estimator of  $\alpha_j=\displaystyle\lim_{n\rightarrow\infty}{n_j(n)/n}$ and  
$$\widehat{\mu}_{j2}(\gamma_0)=I(f) \dfrac{1}{n_j(n)}\sum_{t={t_{j-1}}}^{t_j}{\dfrac{1}{V^2(Z_{t-1})}}.$$  
%
\begin{theorem} \label{th2}
 \noindent Assume that $(A_1)$-$(A_7)$ hold. Then, for any $\beta \in \mathbb R^{k+1}$,
\begin{itemize}
\item[(i)]  Under $H_0$, $\mathcal{T}_n(\gamma_0,\beta)\overset{D}{\longrightarrow}\mathcal{N}( 0,1),$ as $n\rightarrow\infty.$
\item[(ii)] Under $H^{(n)}_{\beta}$, at level of significance $\alpha\in(0,1)$, the asymptotic power of the test based on $\mathcal{T}_n(\gamma_0,\beta)$ is $ \mathcal P_{k,t^k}= 1-\Phi\left(z_\alpha -\varpi(\gamma_0,\beta)\right)$, where $z_\alpha$ is the $(1-\alpha)$-quantile of a standard Gaussian distribution with cumulative distribution function $\Phi$.
\item[(iii)] The test based on $\mathcal{T}_n(\gamma_0,\beta)$ is locally asymptotically optimal.
\end{itemize}
\end{theorem}
\begin{proof}
See Appendix.
\end{proof}
\subsection{The parametric models}
Now, we place ourselves in the framework of the model (\ref{mod}) with the functions $ T $ and $ V $ of known forms but depending on the unknown parameters. More precisely, we assume that $T (x) =T_\rho (x),$ $V (x) = V_\theta (x),$ $\rho\in\Theta\subset\mathbb R^l$ and $\theta\in\widetilde{\Theta}\subset\mathbb R^q.$ Let $\psi_0=(\rho_0^{\top},\theta_0^{\top})^{\top}\in\Theta \times \widetilde{\Theta}\subset \mathbb R^l \times \mathbb R^q$ the true nuisance parameter of the model (\ref{mod}). For $\gamma\in\mathbb R^{k+1}$ and $\psi=(\rho^{\top},\theta^{\top})^{\top}\in\Theta \times \widetilde{\Theta},$ define 
\begin{equation} \label{epsi1}
 \varepsilon_t(\psi,\gamma)=\dfrac{X_t-T_\rho(Z_{t-1})-\gamma^{\top}\omega(t)}{V_\theta(Z_{t-1})},  t\in\mathbb Z.
 \end{equation}
We make the following additional assumptions:
\begin{itemize}
\item[$(B_0)$:] For any $\theta\in\widetilde{\Theta}$ and  $z\in\mathbb R^p,$ $V_\theta(z)>\tau,$ where $\tau$ is a  positive real number.
\item[$(B_1)$:] There exists $ \delta\geq0$ such as $\displaystyle\int_{\mathbb R^p}{||x||^{2+\delta}}dF_j(x)<\infty,$ $j=1,\ldots,k+1.$
\item[$(B_2)$:] $\displaystyle\int_{-\infty}^\infty{|x\phi'_f(x)|f(x)}dx<\infty.$
\item[$(B_3)$:] Denote $\mbox{Int}(\Theta)$ and $\mbox{Int}(\widetilde{\Theta})$ the interior of  $\Theta$ and $\widetilde{\Theta}$ respectively. The functions $T_\rho(z)$ and $V_\theta(z)$ are continuous and differentiable with respect to $\rho\in \mbox{ Int}(\Theta)$ and $\theta\in \mbox{Int}(\widetilde{\Theta})$ respectively, and there exist finite numbers $r_1$ and $r_2$ such that $\overline{B}(\rho_0,r_1)\subset \text{Int}(\Theta),$ $\overline{B}(\theta_0,r_2)\subset \text{Int}(\widetilde{\Theta}) $, such that the largest number among $\sup_{\rho\in  \overline{B}(\rho_0,r_1)} ||\partial_{\rho} T_\rho(z)||$,  $\sup_{\theta \in \overline{B}(\theta_0,r_2)}{|V_\theta(z)|}$, $\sup_{\theta \in \overline{B}(\theta_0,r_2)} ||\partial_{\theta} V_\theta(z)||$ and $\sup_{\theta \in \overline{B}(\theta_0,r_2)} {||\partial_{\theta}^2V_\theta(z)||_M }$ is bounded by some positive function $\vartheta(z)$ such that 
for any $j=1,\ldots,k+1$, 
$\int_{\mathbb R^p}{\vartheta^3(x)}dF_j(x)< \infty.$
\item[$(B_4)$:] The true parameter $\psi_0=(\rho^{\top}_0,\theta^{\top}_0)^{\top}$ has a consistent estimator $\psi_n=(\rho^{\top}_n,\theta^{\top}_n)^{\top}$ satisfying
 $$n^{\frac{1}{2}}(\psi_n-\psi_0)= n^{-\frac{1}{2}}\sum_{t=1}^n{} \Psi (Z_{t-1},\psi_0) \Omega^{\top}[\varepsilon_t(\psi_0,\gamma_0)]+ o_P(1),$$
 $$\Psi(z,\psi_0)=(\Psi^{\top}_1(z,\psi_0),\Psi^{\top}_2(z,\psi_0))^{\top}, \ \ \ \Omega(z)=(\Omega_1(z),\Omega_2(z))^{\top}  \ z\in\mathbb R^p, $$
 $$ \Psi_m(z,\psi_0)=(\Psi_{m1}(z,\psi_0),\ldots, \Psi_{ml}(z,\psi_0))^{\top}\in\mathbb R^{l}, \ m=1,2, \ \  \Omega_1(z) \in \mathbb R, \ \  z\in\mathbb R^p,$$
$\int_{\mathbb R^p} ||\Psi(x,\psi_0)||^{2+\delta}dF_j(x)<\infty,$
 $\int_{\mathbb R}{||\Omega(x)||^{2+\delta}f(x)}dx<\infty$, $\int_{\mathbb R}{\Omega(x)f(x)}dx=\mathbf{0}.$
  \end{itemize}
\begin{remark}
Assumptions ($B_0$)-($B_3$) are satisfied by usual models as parametric AR, ARCH, EXPAR, TARCH models with Gaussian noise. For these models, assumption ($B_4$) is satisfied by usual estimators as least-squares or pseudo-likelihood estimators.
\end{remark}

%
\begin{proposition} \label{pr1}
Assume that $(A_1)$-$(A_7)$, $(B_4)$ hold. Then as $n\rightarrow\infty,$
\begin{itemize}
\item[(i)] Under $H_0$, $$\sqrt{n}(\psi_n-\psi_0)\overset{D}{\longrightarrow}\mathcal{N}(0,\Sigma),$$
\item[(ii)] For any $\beta \in \mathbb R^{k+1}$, 
under $H^{(n)}_{\beta}$, $$\sqrt{n}(\psi_n-\psi_0)\overset{D}{\longrightarrow}\mathcal{N}(\nu,\Sigma)$$
\end{itemize}
$$\nu:=\displaystyle\int_{\mathbb R}{\Omega^{\top}(x)\phi_f(x)f(x)}dx\sum_{j=1}^{k+1}{\alpha_j\beta_j\displaystyle\int_{\mathbb R^p}{\dfrac{\Psi(x,\psi_0)}{V_{\theta_0}(x)}}dF_j(x) }$$
$$\Sigma=\begin{pmatrix} 
\Sigma_{11}& \Sigma_{12} \\ 
\Sigma_{21}  & \Sigma_{22}
\end{pmatrix}$$
\begin{eqnarray*}
\Sigma_{ml}&: =&\int_{\mathbb R}{\Omega_m(x)\Omega_l(x)f(x)}dx\sum_{j=1}^{k+1}{\alpha_j}\int_{\mathbb R^p}{\Psi_m(x,\psi_0)\Psi^{\top}_l(x,\psi_0)}dF_j(x), \ \ m,l=1,2.
\end{eqnarray*}
\end{proposition}
\begin{proof}
See Appendix.
\end{proof}
\subsubsection{The parameter $\gamma_0$ is known}

\noindent
In practice, the case where the parameter $\gamma_0$ is known may be encountered when there is no apparent change, and one wishes to test for possible  weak changes. 
That is the situation where $\gamma_0=0$. This is what is usually tested in the literature.
 \smallskip

\noindent
For any $\beta \in\mathbb R^{k+1},$  denote by 
$ \Lambda_n(\psi_0,\gamma_0,\beta)$ the log-likelihood ratio of $H_0$ against $H^{(n)}_{\beta}$ and by $\Delta_n(\psi_0, \gamma_0, \beta)$ the counterpart of $\Delta(\gamma_0, \beta)$ given by (\ref{cent}) with $\varepsilon_t(\gamma_0)$ defined by (\ref{epsi}), replaced by $\varepsilon_t(\psi_0, \gamma_0)$ defined by (\ref{epsi1}).
For all $l \le 3$ and $j=1, \ldots, k+1$, define the following real numbers 
$$\mu_{jl}(\psi_0, \gamma_0)=I(f) \int_{\mathbb R^p} {1 \over V_{\theta_0}^l(x)} dF_j(x)$$
$$\mu(\psi_0,\gamma_0,\beta)=\sum_{j=1}^{k+1} \alpha_j \beta_j^2 \mu_{j2}(\psi_0,\gamma_0)=\varpi^2(\psi_0,\gamma_0,\beta).$$
\noindent 
Note that, since  $V_{\theta}(x) > \tau>0$, for any $x$ and $\theta$, the $ \mu_{j\ell}(\psi_0,\gamma_0)$'s are finite. Next, 
since for any $j=1,\ldots,k+1,$ $\mu_{j2}(\psi_0, \gamma_0)$ depends on $F_j$ which itself depends on $\psi_0$ (which is unknown) and on $\gamma_0$, we estimate it by $\widehat{\mu}_{j2}(\psi_n, \gamma_0)$ given by 
$$\widehat{\mu}_{j2}(\psi_n, \gamma_0)=I(f) \dfrac{1}{n_j(n)}\sum_{t={t_{j-1}}}^{t_j}{\dfrac{1}{V_{\widehat{\theta}_n}^2(Z_{t-1})}}.$$
So, although we can consider any consistent estimators of $\mu(\psi_0, \gamma_0,\beta)$ and $\varpi(\psi_0, \gamma_0,\beta)=\mu^{1 \over2}(\psi_0, \gamma_0,\beta)$,  we will take them here to be respectively, 
$$\widehat{\mu}_n(\psi_n,\gamma_0,\beta)=\sum_{j=1}^{k+1}{\widehat{\alpha}_j\beta^2_j\widehat{\mu}_{j2}(\psi_n,\gamma_0)} \ \ \text{  and    } \ \ \widehat{\varpi}_n(\psi_n,\gamma_0,\beta)=\widehat{\mu}_n^{\frac{1}{2}}(\psi_n,\gamma_0,\beta), $$ 
where for all $ j=1,\ldots,k+1,$ $\widehat{\alpha}_j$ is an estimator of  $\alpha_j$ which can be taken to be $\widehat{\alpha}_j={n_j(n)/n}.$
\begin{proposition} \label{pr2}
Assume that $(A_1)$-$(A_7)$, $(B_0)$-$(B_4)$ hold. Then, for any sequence of  consistent and asymptotic normal estimators $\lbrace\psi_n\rbrace_{n\geq 1}$ of $\psi_0$, under $H_0$, as $n\rightarrow\infty,$ for any $\beta \in \mathbb R^{k+1}$, we have
\begin{itemize}
\item[(i)]
$$\Delta_n(\psi_0,\gamma_0,\beta)=\Delta_n(\psi_{N(n)},\gamma_0,\beta)+o_P(1)$$
\item[(ii)] $$\widehat{\varpi}_n(\psi_n,\gamma_0,\beta)\longrightarrow{\varpi}({\psi}_0,\gamma_0,\beta),$$
\end{itemize}
where $\lbrace N(n)\rbrace_{n\geq 1}$ stands for a subset $\lbrace 1,\ldots,n\rbrace$ such that $n/N(n)\longrightarrow 0$ as $n\rightarrow\infty$.
\end{proposition}
\begin{proof}
See Appendix.
\end{proof}
\medskip

\noindent For any $\beta \in \mathbb R^{k+1}$, for testing $H_0$ against $H^{(n)}_{\beta}$, we consider the statistic 
$$ \mathcal{T}_n(\psi_{N(n)},\gamma_0,\beta)= \dfrac{\Delta_n(\psi_{N(n)},\gamma_0,\beta)}{\widehat{\varpi}_n(\psi_{N(n)},\gamma_0,\beta)}.$$
%
\begin{theorem}  \label{th3}
\noindent Assume that $(A_1)$-$(A_7)$, $(B_0)$-$(B_4)$ hold. Then, for any $\beta \in \mathbb R^{k+1}$,
\begin{itemize}
\item[(i)]  Under $H_0$, $\mathcal{T}_n(\psi_{N(n)},\gamma_0,\beta)\overset{D}{\longrightarrow}\mathcal{N}( 0,1),$ as $n\rightarrow\infty.$
\item[(ii)]  Under $H^{(n)}_{\beta}$, at level of significance $\alpha\in(0,1)$, the asymptotic power of the test based on $\mathcal{T}_n(\psi_{N(n)},\gamma_0,\beta)$  is $\mathcal P_{k,t^k}= 1-\Phi\left(z_\alpha -\varpi(\psi_0,\gamma_0,\beta)\right)$, where $z_\alpha$ is the $(1-\alpha)$-quantile of a standard Gaussian distribution with cumulative distribution $\Phi.$
\item[(iii)] The test based on $\mathcal{T}_n(\psi_{N(n)},\gamma_0,\beta)$ is locally asymptotically optimal.
\end{itemize}
\end{theorem}
\begin{proof}
See Appendix.
\end{proof}
\subsubsection{The parameter $\gamma_0$ is unknown}

\noindent 
In practice $\gamma_0$ is generally unknown. It has to be estimated, as well as  
$\gamma_n= \gamma_0+\beta/\sqrt{n}$, 
 for the computation of the likelihood-ratio statistic. 
 For any $\beta \in \mathbb R^{k+1}$, let $\widehat{\gamma}_{0,n}$
and $\widetilde{\gamma}_{n}$
 be respectively, the maximum likelihood estimator of $\gamma_0$ and $\gamma_n$. It is easy to check that for larger values of $n$, in probability, 
$$\widetilde{\gamma}_{n}=\widehat{\gamma}_{0,n}+\dfrac{\beta}{\sqrt{n}}.$$
%
%

\noindent
Let $N(n)$ be any sub-sequence of $\lbrace 1,\ldots,n\rbrace$ satisfying $n/N(n)\longrightarrow 0$ as $n\rightarrow \infty.$ For any $\beta \in \mathbb R^{k+1}$, our test statistic for this testing problem is 
$$\mathcal{T}_n(\psi_{N(n)},\widehat{\gamma}_{0,N(n)},\beta)=\dfrac{\Delta_n(\psi_{N(n)},\widehat{\gamma}_{0,N(n)},\beta)}{\widehat{\varpi}_n(\psi_{N(n)},\widehat{\gamma}_{0,N(n)},\beta)}.$$
\begin{proposition} \label{pr3}
Assume that $(A_1)$-$(A_7)$, $(B_0)$-$(B_4)$ hold. Then, for any sequence of consistent and asymptotically normal estimators $\lbrace (\psi_n,\widehat {\gamma}_{0,n})\rbrace_{n\geq 1}$ of $(\psi_0,\gamma_0)$, as $n\rightarrow\infty,$ under $H_0$, for any $\beta \in \mathbb R^{k+1}$,
$$\Delta_n(\psi_0,{\gamma}_0,\beta)=\Delta_n(\psi_{N(n)},\widehat{\gamma}_{0,{N(n) }},\beta)+o_P(1).$$
\end{proposition}
\begin{proof}
See Appendix.
\end{proof}
\medskip

%
\begin{theorem} \label{th4}
Assume that $(A_1)$-$(A_7)$, $(B_0)$-$(B_4)$ hold. Let $N(n)$ be any  sub-sequence of $\lbrace 1,\ldots,n\rbrace$ satisfying $n/N(n)\longrightarrow 0$ as $n\rightarrow\infty.$ Then for any $\beta \in \mathbb R^{k+1}$,
\begin{itemize}
\item[(i)]  Under $H_0$, $\mathcal{T}_n(\psi_{N(n)},\widehat{\gamma}_{0,N(n)},\beta)\overset{D}{\longrightarrow}\mathcal{N}( 0,1),$ as $n\rightarrow\infty.$
\item[(ii)] Under $H^{(n)}_{\beta}$, at level of significance $\alpha\in(0,1)$, the asymptotic power of the test based on $ \mathcal{T}_n(\psi_{N(n)},\widehat{\gamma}_{0,N(n)},\beta)$ is $ \mathcal P_{k,t^k}=1-\Phi\left(z_{\alpha} -\varpi(\psi_0,\gamma_0, \beta)\right)$, where  $z_{\alpha}$ is the $(1-\alpha)$-quantile of a standard Gaussian distribution with cumulative distribution $\Phi$.
\item[(iii)] The test based on $ \mathcal{T}_n(\psi_{N(n)},\widehat{\gamma}_{0,N(n)},\beta)$  is locally asymptotically optimal.
\end{itemize}
\end{theorem}
\begin{proof}
This theorem is a straightforward consequence of Proposition \ref{pr3} and Theorem \ref{th3}.
\end{proof}
\section{Practical considerations}
\noindent 
In this section, we apply our theoretical results on simulated data, using the software R. We focus on the study of the power as a function of $\beta$ when the break locations are fixed, and as a function of the breaks locations when the associated $\beta$ is estimated. The breaks are estimated by the strategies described in subsection 2.3. 
The results we present in the sequel  are obtained for the nominal level $\alpha=5\%$. We do not present those obtained  for $\alpha=1\%, 10\%$ as they are very similar. Almost all the estimators in this section are computed from 5000 replications.
\subsection{Power study for given breaks locations}
We start with the study of the theoretical local power of our test for known breaks. That is, we consider the testing problem stated in Section 2 when the $t_j$'s are assumed to be given. Thus, the power of the test is a function of $\beta$ only. In this subsection, we study the behavior of this function for several models from the more general following one:
\begin{eqnarray} \label{sim}
X_t &=& \left(\rho_1+\rho_2X_{t-1}e^{-\rho_3 X_{t-1}^2}\right)X_{t-1}+\left(\gamma_0+ \dfrac{\beta}{\sqrt{n}} \right)^{\top} \omega(t) 
+\left(\theta_1 + \theta_2 X_{t-1}^2 e^{-\theta_3 X_{t-1}^2} \right)^{\frac{1}{2}}\varepsilon_t,  t\in\mathbb Z
\end{eqnarray}
where the $\rho_j$'s, $\theta_j$'s and $\gamma_0$ are parameters to be specified in each particular model considered, $n$ is the sample size, $(\varepsilon_t)_{t\in\mathbb Z}$ is a standard white noise with a differentiable density $f$, and $\beta \in [-10,10]^{k+1}$, with $k$ standing for the number of breaks.
\begin{center}
\begin{figure}[h!]
   \begin{subfigure}[b]{0.5\linewidth}
      \centering \includegraphics[scale=0.3]{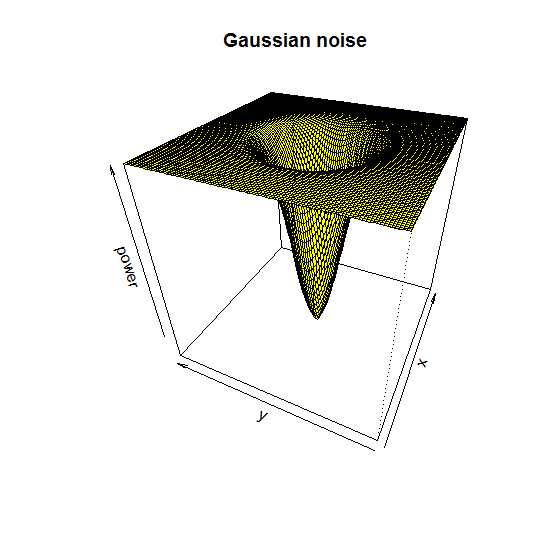}
      \caption{ \large $\rho_1=0.5,$ $\rho_2=0,$ $\theta_1=1,$ $\theta_2=0$}
   \end{subfigure}\hfill
   \begin{subfigure}[b]{0.5\linewidth}   
      \centering \includegraphics[scale=0.3]{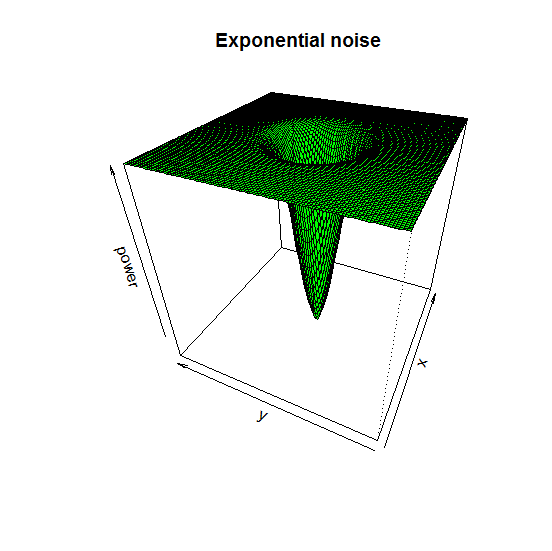}
      \caption{\large  $\rho_1=0.5$, $\rho_2=0,$ $\theta_1=1,$ $\theta_2=0$}
   \end{subfigure}\\
   \begin{subfigure}[b]{0.5\linewidth}
      \centering \includegraphics[scale=0.3]{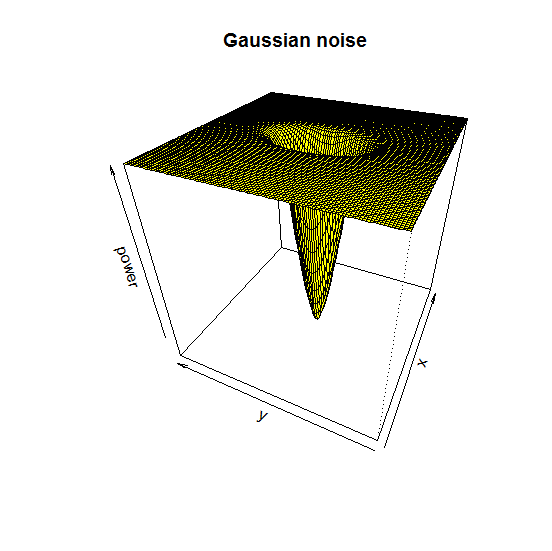}
      \caption{\large $\rho_1=0.5,$ $\rho_2=0.2,$ $\rho_3=50,$ \\
$\theta_1=0.1,$ $\theta_2=0.0025,$ $\theta_3=1$}
      \end{subfigure} \hfill
      \begin{subfigure}[b]{0.5\linewidth}
      \centering \includegraphics[scale=0.3]{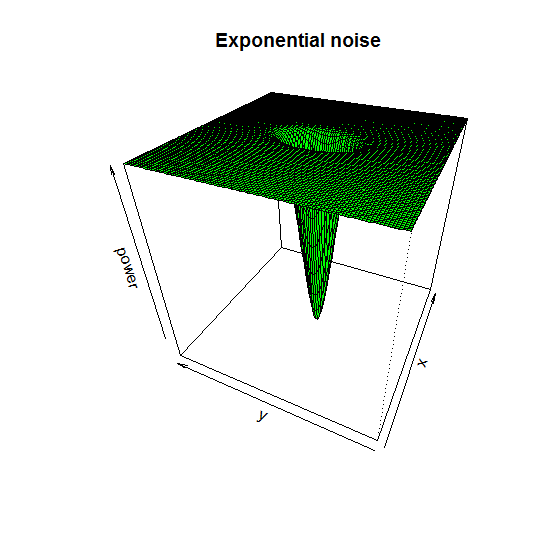}
      \caption{\large $\rho_1=0.5,$ $\rho_2=0.2,$ $\rho_3=50,$ \\ $\theta_1=0.1,$$ \theta_2=0.0025,$ $\theta_3=1$}
      \end{subfigure}\hfill
\caption[]{Power of the test for standard Gaussian and standardized exponential noises.}
\label{f2}
\end{figure}
\end{center}
\subsubsection{The parameter $\gamma_0=0$}
Here we consider $\gamma_0=0$ and we test no break against one single break, and against two breaks.
\medskip 

\noindent
$\bullet$ Case $k=1$
\smallskip 

\noindent
We illustrate our results with the model (\ref{sim}) for $\rho_1 = 0.5$, $\rho_2=0$, $\theta_1= 1, \theta_2=0 $, $ n = 60 $, $t_1 = 30$ and $t_2 = 60$.
Figure \ref{f2} (a) and Figure \ref{f2} (b) present the power of the test as a function of $\beta \in[-10,10]^2$,  for $f$ respectively the standard normal density, and the standardized exponential density with rate parameter 1.25. Note that in the latter situation, the 
$\varepsilon_t$'s are obtained by taking $ \varepsilon_t = \lambda (\widetilde \epsilon_t- 1/\lambda) $ with $\widetilde \epsilon_t$ being an exponential random variable with parameter $ \lambda$. 
It can be seen that on these two graphics, the power tends to 1 and the power corresponding to the exponential noise tends faster to 1 than that of the Gaussian noise. 
Other trials involving the above model with various exponential density function gave very similar results. 
\medskip

We next consider (\ref{sim}) for $\rho_1 = 0.5, \rho_2= 0.2, \rho_3= 50$, $\theta_1=0.1, \theta_2= 0.0025, \theta_3=1$, $ n = 60 $, $t_1=30$ and $t_2= 60$.
Figure \ref{f2} (c) and Figure \ref{f2} (d) show the power, respectively for standard normal density $f$ and standardized exponential density $f$ with parameter 1.25. It can be seen that in either case, it increases  with  $\beta$ to 1. Here also, other trials involving the above model with various exponential density function gave very similar results. 
\medskip 

\noindent
$\bullet$ Case $k=2$
\smallskip 

\noindent
For the case of two breaks, we took $ n = 100,$ $ t_1=30$, $t_2=60$  and $ t_3= 100$, $\rho_1=0.5, \rho_2=0$, $\theta_1=1, \theta_2=0$. 
 The results for the power corresponding to the standard Gaussian $f$ ($\mathcal N(0,1)$) and the standardized exponential $f$  ($\mathcal E (0,1)$) are displayed in Table \ref{tab:1} below. It is seen that the power moves toward 1 when $\beta$ moves away from $(0,0,0)^{\top}$.
%
%
\medskip

\noindent
$\bullet$ 
We also studied the case of 3 breaks. The results are not reported as they 
were very similar to the case of 2 breaks. 
\subsubsection{The parameter $\gamma_0$ and $f$ are unknown}
\noindent In practice $\gamma_0$ and $ f $ may be unknown and may need to be estimated.
As the theory is more complex in this case we have not tried to tackle it. But we have done some trials in order to have an idea on how the test may behave in this situation.
 While we estimated $\gamma_0$ by a maximum likelihood method, $f$ was estimated by the well-known kernel (or the Parzen-Rosenblatt) estimator $\widehat{f}_n$ defined by
$$\widehat{f}_n(x)=\dfrac{1}{nh_n}\sum_{t=1}^n{K\left( \dfrac{x-\widehat{\varepsilon}_t}{h_n} \right),\quad x\in\mathbb R,
}$$
where 
$$
 \widehat{\varepsilon}_t=\dfrac{X_t-T_{\rho_n}(Z_{t-1})-\widehat{\gamma}_n^{\top}\omega(t)}{V_{\theta_n}(Z_{t-1})},$$
with  $\psi_n=(\rho_n,\theta_n)$ and $\widehat{\gamma}_n$ being the maximum likelihood estimates of the nuisance parameter $\psi_0=(\rho_0,\theta_0)$ and $\gamma_0$ respectively, 
 $h_n=n^{-1/5}$ the bandwidth and $K$ the Gaussian kernel function.
%

The power of the test was computed based on these estimates which are themselves based on observations  from (\ref{sim}) for $\rho_1= 0.5, \rho_2=0$ and $\theta_1 = 1, \theta_2=0 $, and the standard Gaussian $f$ and standardized exponential $f$ for 3 different rate parameters.
%
The graphics of the corresponding powers are not presented as they were very similar to those on Figure \ref{f2}. 

\begin{table}
%
\scalebox{0.7}{\begin{tabular}{|c|cccccccccccc|}
\hline
$\beta$ & $\begin{pmatrix} 
0.2\\
 0.25\\
 0.5\end{pmatrix}$ & $\begin{pmatrix} 
0.5\\
 0.75\\
 1\end{pmatrix}$& $\begin{pmatrix} 
0.75\\
 1\\
 1.5\end{pmatrix}$ & $\begin{pmatrix} 
0.75\\
 1.2\\
 2\end{pmatrix}$ & $\begin{pmatrix} 
2.2\\
 1.5\\
 2.2\end{pmatrix}$ & $\begin{pmatrix} 
2.5\\
 2\\
 3\end{pmatrix}$ & $\begin{pmatrix} 
2.75\\
 2.2\\
 3.2\end{pmatrix}$ & $\begin{pmatrix} 
3\\
 2.5\\
 3.8\end{pmatrix}$ & $\begin{pmatrix} 
3.5\\
 3\\
 4\end{pmatrix}$ & $\begin{pmatrix} 
3.5\\
 3\\
 0.75\end{pmatrix}$& $\begin{pmatrix} 
4\\
 3.5\\
 4\end{pmatrix}$ & $\begin{pmatrix} 
4.7\\
 4\\
 3\end{pmatrix}$ \tabularnewline
\hline
$\mathcal N(0,1)$ & 0.119 & 0.271  &0.447 & 0.607&   0.831  &0.955&    0.975 &  0.994    &0.999  &  0.953   & 0.999&  1.000

 \tabularnewline

$\mathcal E (0,1)$ & 0.144 &  0.363&   0.596&   0.773 &  0.946   &0.994 &  0.998  & 0.999&    0.999 &   0.994 &   1.000& 1.000\\
\hline
\end{tabular}}
\caption{Power against an AR(1) model with 2 breaks}
\label{tab:1} 
\end{table}
\subsection{Change-points detection and their locations estimation}
In this subsection, we investigate change-points detection and their locations estimation 
 in simulated data, and in three series of real data. 
%
%
%
\medskip

\noindent
As our methods are based on change magnitudes estimates, change location estimates may not be accurate on one sample. Then it can be interesting to consider the mean of these estimates based on a large number of samples. 
Recalling that the first step of our methods is to get potential information from the chronogram of the series at hand, a problem arises as the samples may have different potential change-points, and different number of change location estimates.
%
To overcome this difficulty, in the first phase of our investigations, we assumed the areas of the changes to be known (not changes themselves). In the second phase, we did no more make this assumption.

%
%
\subsubsection{Simulated data}
The data $X_1, \ldots, X_n$, $n=200$ were simulated from model (\ref{mod}) for $V(x)=1$ and either $T(x)=0$ or $T(x)= 0.5 x$, and for several $\gamma_0$, $\beta$ and $t_j, \ j=1, \ldots, k$. In either case, the noise was Gaussian. In this phase, the areas of changes were assumed to be known and the change-point locations were estimated on the basis of 5 000 replications. We took $\zeta=0.01$, $m=20$ and $h=20$.
%
%
%
%
%

We first studied the situation where the data present no break, and we tested for no break against the existence of a break. The $n=200$ observations were simulated for $T(x)=0, \ V(x)=1$ and $\gamma_0=\beta=0$. Thus, they were stationary and contained no break. We simulated four sets of $n=200$ observations. For each of them, we applied the strategy S1 with assumed one single break location at $t_1$, for $t_1=30, \ldots, 170$. As can be seen from Figure \ref{f3} (a), the curve of the power function, as a function of $t_1$, is {\it flat}. Its values are comprised between 5\% and 5.3\%, which shows that our method did not detect any significant change in the data got from models with no break. 

We next considered the case the data contained one single weak break-point 
 and we tested for no break against one single weak break. 
In this case, we used Part (ii) of Remark 4.
The data were computed from (\ref{mod}) for $V(x)=1$ and $T(x)= 0.5 x $, $\gamma_0=(0,0)$ and $\beta=(0,\beta_2)$, for $\beta_2= -0.75$ $,-0.6,$ $-0.5,$ $0.4,$ $0.5,$ $0.8$, respectively combined with the respective breaks locations $t_1=60, 90, 120$ and $160$.  In any case, we applied the strategy S2. Figure \ref{f3} (b) shows that the power is significantly larger than the nominal level 5 \% and, generally, reaches its maximum at the right break locations. 

Another case studied is that where there were 
 two weak break-points in the data. 
The test was for no break against two weak breaks. We assumed known the areas of the breaks. 
The data were computed from (\ref{mod}) for $V(x)=1$ and either $T(x)= 0.5 x, \ T(x)=0$, $V(x)=1$, $\gamma_0=(0,0,0)$ and $\beta$$=(0,\beta_2,\beta_3)$ for $(\beta_2, \beta_3)=(0.5,-0.5),$ $ (0.4,0.4),$ $ (-0.3,0.3),$  $(0.75,-0.75),$ $ (-0.2,0.2)$ and $(0.8,-0.8)$, respectively combined with the respective couples of breaks locations $(t_1,t_2)=(40, 120)$ and $(50,140)$.  We applied the strategy S2 with $\mathcal S_2$$=$$\{30, \ldots, 50 \}$$ \times$$ \{110, \ldots, 130 \}$ for $(t_1,t_2)$=$(40, 120)$ and $\mathcal S_2$$=$$ \{40, \ldots, 60 \}$$\times$$ \{130, \ldots, 150 \}$ for $(t_1,t_2)$=$(50, 140)$.
From Figure \ref{f3} (c) and Figure \ref{f3} (d), it can be seen that  the power of the test deviates from the nominal level and generally reaches its maximum at the right couples of the break locations listed above. Table 2 gives more extensive results. 
As can be seen there, our methods generally detect changes and their correct locations in the simulated data. 
%
%

\begin{table}
\begin{center}
%
\scalebox{.7}{\begin{tabular}{|lr|ccccccc|}
\hline
%
&{$\beta$}& $\begin{pmatrix}
0\\
0.2\\
-0.2\end{pmatrix}$&  $\begin{pmatrix}
0\\
-0.4\\
0.4\end{pmatrix}$&  $\begin{pmatrix}
0\\
-0.5\\
0.5\end{pmatrix}$&     $\begin{pmatrix}
0\\
0.7\\
-0.7\end{pmatrix}$&   $\begin{pmatrix}
0\\
0.9\\
-0.9\end{pmatrix}$ &   $\begin{pmatrix}
0\\
1\\
-1\end{pmatrix}$ &    $\begin{pmatrix}
0\\
1.3\\
-1.3\end{pmatrix}$      \\ 
$(t_1,t_2)$ &&&&&&&& \\ \hline
( 60,100)& &(59,100) &(70,100) & (60,100) & (65,100)& (60,100) & (60,100)& (60,100)\\ 
(40,120) && (38,120) & (40,120)& (40,120) & (40,120) & (40,120) & (40,120) &(40,120)\\ 
(50,140)&&(45,140) & (50,140) & (50,140) &(50,138)& (50,141)& (50,140) & (50,140)\\ 
(80,160)&&(80,160) &  (80,160) & (80,159)& (80,159)& (80,160) & (80,160) &(80,160)\\ 
(100,170)& &(99,170)& (100,170) & (100,170) & (101,170) & (100,169)& (100,170)& (100,170)\\ \hline
\end{tabular}}
\caption{Location estimates on the basis of 5000 replications using S2, for 2 simulated breaks.}
\label{tab:2} 
%
\end{center}
\end{table}

We compared our methods denoted CURRENT-S2 and CURRENT-S3 with 
those studied in \cite{laj3}, denoted here by SCUSUM (Standard CUSUM) and RCUSUM (R\'enyi CUSUM), and with WBS, NOT and ID studied respectively in \cite{fryz}, \cite{bara} and \cite{cho}. 
For performing the SCUSUM and the RCUSUM, we used the R routine {\it CUSUM.test} from the library {\it CPAT}. For WBS we used the routine {\it wbs} from the library {\it wbs}. For NOT we used the function {\it not} from the library {\it not}. For ID, the routine used was {\it ID}, from the library {\it IDetect}.

We started with series simulated with no change. Here, while our strategy S1 never found any change in all the generated series, SCUSUM, RCUSUM, WBS and sometimes NOT each found a change in almost all of these series. ID even found more than one change in some of them.

For the simulated data with one single change, the location estimates obtained with our methods using Part (ii) of Remark 4, were generally more accurate (see Table \ref{tab:3}). SCUSUM and RCUSUM were able to provide reasonable estimates only for changes with large magnitude. The RCUSUM, performed in \cite{laj3} to estimate change-point locations in the earliest and the latest observations, is beaten by our methods in estimating such locations on the example studied.  WBS, NOT and ID generally produced sequences of change-point locations some of which containing values close to the correct locations. 
Their corresponding results in Table \ref{tab:3} are based on one single set of simulated data while those of CURRENT-S2, CURRENT-S3, SCUSUM and RCUSUM are based on 5000 replications. Indeed, as WBS, NOT and ID sometimes gave more than one point, there was no way to evaluate their mean as their number were different from one replication to another.

We finally consider the situation where neither the number nor the 
areas of the changes are assumed to be known. These were obtained from the chronograms. We compared our methods with SCUSUM, RCUSUM, WBS, NOT and ID for the case of one single change. In the case of two changes the comparison is with WBS, NOT and ID. As can be sen from Table \ref{tab:4}, in the case of one break, for small magnitudes ($\beta_2< 15$), the estimates from our methods are globally more close to the true locations than the others, even for end-points-samples. The same can be said for larger magnitudes ($\beta >15$). For two breaks, As can be sen from Table \ref{tab:5}, again for weak magnitudes, although they sometimes under/over estimate the number of breaks, they are generally able to find at least one break, while the other methods fell to find anyone. For significant magnitudes our methods generally still be better.

\begin{table}
\begin{center}
\scalebox{.7}{\begin{tabular}{|c|cccc|}\hline
True break & 30 & 60 &120& 170\tabularnewline
\hline
$\beta=(0,\beta_2)$ & $(0,-0.5)$ & $(0,0.4)$& $(0,0.5)$ & $(0,-0.75)$
 \tabularnewline
\hline
CURRENT-S2 & 30 &  60 & 120&170  \tabularnewline
SCUSUM & 99 &  99 &100 & 100 \tabularnewline
RCUSUM & 91 &  60 & 143  &126\tabularnewline
 WBS & 71;76 & 17;55;72;124;170&-&105;111;138;143 \\
 NOT & 71;76 & 17;55;72;114;124&-&105;111;138;143 \\
ID & 30;58;76;107;110 &62;72;114;2124;140 &96;101;106;122;129 &105;111;138;143 \\
\hline
$\beta=(0,\beta_2)$ & $(0,-5)$ & $(0,4)$& $(0,5)$ & $(0,-2)$
 \tabularnewline
\hline
CURRENT-S2  & 30 &  60 &  120 & 171  \tabularnewline
SCUSUM & 76 &  81 &  113&105 \tabularnewline
RCUSUM & 29 &  63 &80 & 44\tabularnewline
 WBS & 31;44 &71 &123 &113 \\
 NOT & 31;44;175;182 &71 &32;84;100;112;158 &113 \\
ID &31;47;50;58;63;198  &36;38;95;99;113 &111;123;142;159 &95;106;113;139;147 \\
\hline
$\beta=(0,\beta_2)$ & $(0,-15)$ & $(0,14)$& $(0,15)$ & $(0,-12)$
 \tabularnewline
\hline
CURRENT-S2  & 30 &  60 &  120 &170  \tabularnewline
SCUSUM & 43 &  64 &119 & 153 \tabularnewline
RCUSUM & 35 &  77 &  121 &150\tabularnewline
WBS & 31 &68;98;128;149;190 &39;122 &10;61 \\
 NOT & 32;60;63;138;144 &52;86;98;128;188 &27;39;120 &10;61 \\
ID &32;37;63;96;144;179  & 45;52;86;98;128& 112;115;123;166 &103;124;145;186;195 \\
\hline
$(0,\beta_2)$ & $(0,-25)$ & $(0,24)$& $(0,25)$ & $(0,-22)$
 \tabularnewline
\hline
CURRENT-S2 & 30 &  60 &  120 &170  \tabularnewline
SCUSUM & 36 &  62 &  120 &163 \tabularnewline
RCUSUM& 38 &  61& 121 &166\tabularnewline
WBS &32  &64;129;164 &32;92;110;172;180 &171;191 \\
NOT & 14;30;41;122 &64.129;164&92;120;126;148;158 &171;191 \\
ID & 9;30;52;108;111;198 &53;68;124;129;188 &108;120;126;148;190 &100;122;150;178;190 \\
\hline
$\beta=(0,\beta_2)$ & $(0,-35)$ & $(0,34)$& $(0,35)$ & $(0,-32)$
 \tabularnewline
\hline
CURRENT-S2  & 30 &  60&  120 &170  \tabularnewline
SCUSUM& 34 &  61 &  120 & 166 \tabularnewline
RCUSUM & 30 &  60 & 121 &170\\
WBS &31;90;127  &60;80 & 121;166;173&26;172 \\
 NOT &30;90;127  &60;80 &65;93;121;172;196 &26;159;172 \\
ID& 30;85;128;158;177 & 21;60;125;156;183 &59;67;76;150;188 &85;140;163;170;173 \\
\hline
\end{tabular}}
\bigskip

\caption{Change locations estimates obtained by CURRENT-S2, SCUSUM and RCUSUM on the basis of 5000 replications, and by WBS, NOT and ID on the basis of one replication for samples with known break areas.}
\label{tab:3}
%
\end{center}
\end{table}
\begin{table}
\begin{center}
\scalebox{.7}{\begin{tabular}{|c|cccc|}\hline
True break & 30 & 60 &120& 170\tabularnewline
\hline
$\beta=(0,\beta_2)$ & $(0,-0.5)$ & $(0,0.4)$& $(0,0.5)$ & $(0,-0.75)$
 \tabularnewline
\hline
CURRENT-S2 &20,163 & 54 & 134  & 171 \tabularnewline
CURRENT-S3 &26,46 & 65  & 116,146 & 65, 175 \tabularnewline
SCUSUM & 67 &  66&  151   &75 \tabularnewline
RCUSUM & 67 &  66  &  151   &78 \tabularnewline
 WBS &25&  --- &151&---\\
 NOT & 25 &  --- &  151&  ---\\
ID & 25 & 97,104&  2,16,151&---\\
\hline
$\beta=(0,\beta_2)$ & $(0,-5)$ & $(0,4)$& $(0,5)$ & $(0,-2)$
 \tabularnewline
\hline
CURRENT-S2 & 31 &  52 &  121 &166  \tabularnewline
CURRENT-S3 & 28,48,186 & 57 &  116 &47,127,171 \tabularnewline
SCUSUM & 42 &  70&   121    &62\tabularnewline
RCUSUM & 40 &  74 &   122 &62\tabularnewline
  WBS & ---&  114 &  ---   &--- \\
 NOT & ---&  114 &  ---   &---\\
ID& 46,187 &  ---&   108  &---\\
\hline
$\beta=(0,\beta_2)$ & $(0,-15)$ & $(0,14)$& $(0,15)$ & $(0,-12)$
 \tabularnewline
\hline
CURRENT-S2 & 37 &  61  &  123    &174 \tabularnewline
CURRENT-S3 & 27 &  62 &  117    &166  \tabularnewline
SCUSUM & 23 &  64  &  117   &165 \tabularnewline
RCUSUM & 33 &  64 &  117 &173 \tabularnewline
WBS & 23 &  61 &  118&173  \\
 NOT & 23 &  61 &  118&173 \\
ID & 23 & 32,43,61,187&117,191& 148,173\\
\hline
$(0,\beta_2)$ & $(0,-25)$ & $(0,24)$& $(0,25)$ & $(0,-22)$
 \tabularnewline
\hline
CURRENT-S2 & 34,94 &  62 &  128  &170 \tabularnewline
CURRENT-S3 & 30,46,96 &  56 &  122  &30,168 \tabularnewline
SCUSUM & 40 &  67 &  120   &167 \tabularnewline
RCUSUM & 30 &  67 & 120 &170 \tabularnewline
  WBS & 30&  63 &    120 &170 \\
 NOT & 30 &  63 &   120 &170\\
ID & 2,11,30&  63 &  39,40,120,128  &171 \\
\hline
$\beta=(0,\beta_2)$ & $(0,-35)$ & $(0,34)$& $(0,35)$ & $(0,-32)$
 \tabularnewline
\hline
CURRENT-S2 & 31 &  60 &  120   &166 \tabularnewline
CURRENT-S3 & 30 &  60 &  116   &166 \tabularnewline
SCUSUM & 31 &  60  &  120    &160 \tabularnewline
RCUSUM & 31 &  60  & 120   &168\\
  WBS & 31&  60 &    120 &170\\
 NOT & 31 &  60&    120 &170 \\
ID & 31&  60,120,180,187 & 120 &172\\
\hline
\end{tabular}}
\bigskip

\caption{Change location estimates obtained by CURRENT-S2, CURRENT-S3, SCUSUM, RCUSUM, WBS, NOT and ID on the basis of one sample with one break.}
\label{tab:4}
%
\end{center}
\end{table}
\begin{table}
\begin{center}
\scalebox{.7}{\begin{tabular}{|c|cccc|}\hline
True breaks & $t_1=30$, $ t_2=170$ &$t_1=30$, $t_2=100$ & $t_1=100$,    $t_2=170$ &$t_1=100$, $t_2=150$ \tabularnewline
\hline
&&$(0,\beta_2,\beta_3)=(0,0.2,0.5)$ &&
 \tabularnewline
\hline
CURRENT-S2 &20,177 & 35,103& 164  & 161\tabularnewline
CURRENT-S3 &29,166,186 & 28,98  & 31,101,166  & 27,96,150 \tabularnewline
 WBS & --- &  --- &  --- &  ---\\
 NOT & --- &  --- &  --- &  ---\\
ID & 28 & ---&  48,71&118,121\\
\hline
&&$(0,\beta_2,\beta_3)=(0,0.2, -0.3)$ &&
 \tabularnewline
\hline
CURRENT-S2 &176 & 21,85 &  179 &  114,135\tabularnewline
CURRENT-S3 &  27,167,187&  28,98 &  31,98,176 &101,151 \tabularnewline
  WBS & ---&  114 &  ---   &--- \\
 NOT & ---&  114 &  ---   &---\\
ID  & --- &  --- & 54,61,81 &161,180 \\
\hline
&&$(0,\beta_2, \beta_3)=(0,15,-13)$ &&
 \tabularnewline
\hline
CURRENT-S2 &  30,167 &  23,85  &  170 &113,148\tabularnewline
CURRENT-S3 & 33,170 &  33,96  &  102,166 &105,164 \tabularnewline
  WBS & 33,170 &  30,100 &  106,170&99,151  \\
 NOT & 33,170 &  30,100 &  106,170&99,151 \\
ID &33,61,170 &  27,100,191 &100,170& 35,90,151 \\
\hline
&&$(0,\beta_2, \beta_3)=(0,14,-3)$ &&
 \tabularnewline
\hline
CURRENT-S2 & 20,180 &  35,86&  90,180  &103,141 \tabularnewline
CURRENT-S3 &  27,166,186 &  27,96&  46,98,166  &46,96,149 \tabularnewline
  WBS &36,169 &  26,98 &   99,170 &103,150\\
 NOT &36,169 &  26,98 &   99,170 &103,150\\
ID & 36,169,190 &  26,98 &  57,99,154,170  &45,103,141\\
\hline
\end{tabular}}
\bigskip

\caption{Change location estimates obtained by CURRENT-S2, CURRENT-S3, 
WBS, NOT and ID on the basis of one sample with two  breaks.}
\label{tab:5}
%
\end{center}
\end{table}

\subsubsection{Real data}
Now, our methods are applied to detecting changes in three time series $Q_1$, $V_3$ and $V_{15}$ relative to the floods of the Upper Hanjiang River in China, collected from 1950 to 2011. Their length is $n=62$. The first series represents the annual maxima daily discharge measured in $m^3/s$. The second represents the annual maxima 3 day flood volume and the third the annual maxima 15 day flood volume, both measured  in $m^3$.  
The chronogram of $Q_1$ in Figure \ref{f4} (a) and those of $V_3$ (black line) and $V_{15}$ (blue line) in Figure \ref{f4} (b) seem to present a little trend and no apparent seasonality. For each of them, using the R routine ma, we estimated the trend (red line) by a five-order moving average and subtracted them from their corresponding raw series. As the trend is assumed to be smooth, the eventual changes in the series are expected to be found in the residual series. More explicitly, 
letting $(T_t)$ be the estimate of the trend, and letting $(Y_t)$ be either of the raw series, we applied our methods to the residual series $X_t=Y_t-T_t$ (detrended series) and considered the model 
$$X_t=\mu_t + \sigma_t \varepsilon_t, \ 1 \le t \le n,$$
where $\mu_t$ and $\sigma_t$ are piece-wise constant functions of $t=1, \ldots, n$ and $(\varepsilon_t) $ is a standard white noise with a Gaussian distribution.

We considered the above model because the Ljung-Box and Box-Pierce tests (see \cite{bro}) applied to each of the three $(X_t)$ rejected their iid hypothesis. Hence, we considered them as heteroscedastic. Also, the Shapiro-Wilk test applied to several subsets of the residual series suggested they were piece-wise Gaussian. 
The graph of the residual associated with $Q_1$ is presented in Figure \ref{f4} (c), while those associated with the two others are on Figure \ref{f4} (d) (black line for $V_3$ and blue line for $V_{15}$). On each of them, the red line represents the trend. In any case, our historical data was the 20 first observations $X_1, \ldots, X_{20}$, we suspected changes around  1986 and 1998. As minimum distance $h$ we took $h=5$, and $\zeta=.01$. 

The strategy S2 applied with the $\mathcal S_j$'s containing 5 points gave, for $Q_1$, locations 1985 and 1996 and for $V_3$ and $V_{15}$, locations 1984 and 1997.  The strategy S3 gave locations 1986 and 1997 for the three series. 

We next applied SCUSUM and RCUSUM sequentially with the same subsets of data used in our sequential method S3. Both gave the same results. That is, for $Q_1$ and $V_3$, changes  at 1971 and 2000 and for $V_{15}$, changes at 1968 and 2000. These locations estimates are a bit far from our estimates which are close to those obtained by \cite{xiong} who found only one change located at 1987 for $Q_1$, and  
 1985 for both $V_3$ and $V_{15}$. 

We finally applied NOT, WBS and ID to these data. NOT gave 1984 for $Q_1$,1985 for $V_3$ and 1984 for $V_{15}$. WBS and ID gave the same results : 1984 for $Q_1$,1985 for $V_3$ and 1984 and 2009 for $V_{15}$. 

How to evaluate the accuracy of these estimates? The Ankan Reservoir was built in 1982. It has been proved to be among other reservoirs, the one mostly influencing the flow regime of the Upper Hanjiang River. It  started storing water in 1989. One year later, its first generator was operational and in 1992, all its four generators were operational. The entire project was complete by 1995. From these information, it is clear that the estimations obtained with SCUSUM and RCUSUM are far from the period invoked above. Those obtained with our methods as well as those from NOT, WBS and ID fall within this period. 

%
%
%
%
\begin{center}
\begin{figure}[h!]
   \begin{subfigure}[b]{0.5\textwidth}   
      \centering \includegraphics[height=5.cm, width=6.cm,scale=0.2]{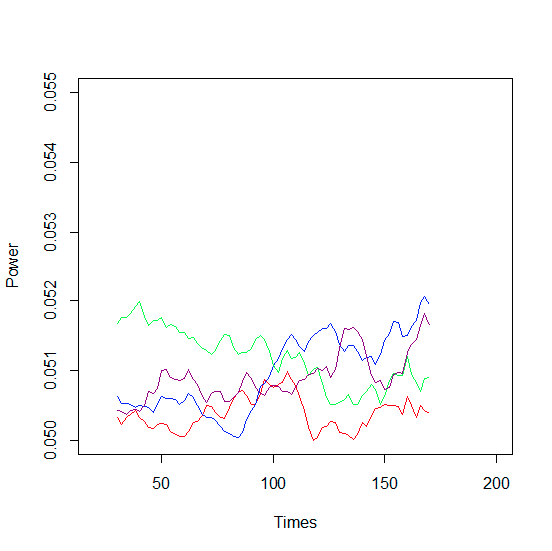}
     \caption{\large \centering AR(1) model}
   \end{subfigure}\hfill
      \begin{subfigure}[b]{0.5\textwidth}
      \centering \includegraphics[height=5.cm, width=6.cm,scale=0.2]{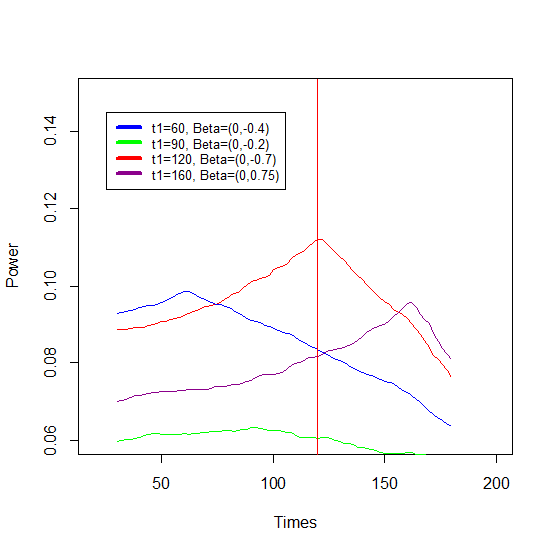}
      \caption{\begin{large}
       \centering AR(1) model with one break\end{large}}
      \end{subfigure}\hfill
      \begin{subfigure}[b]{0.5\linewidth}
      \centering \includegraphics[height=6.cm, width=6.cm,scale=0.3]{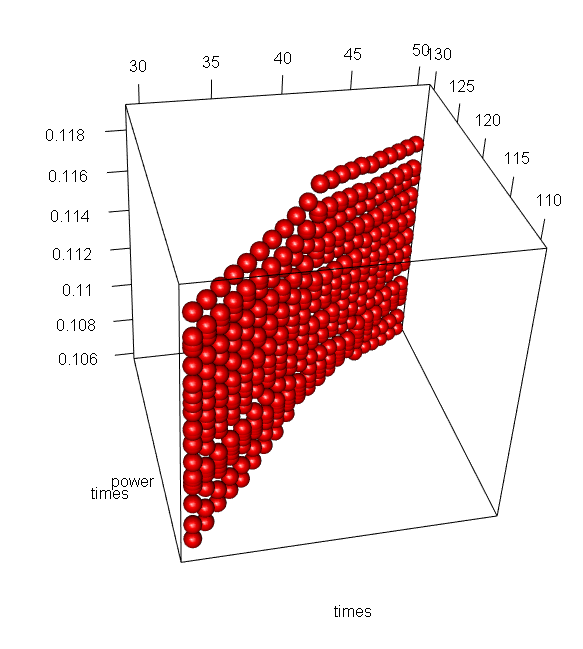}
      \caption{ \large White noise model with \centering $t_1=40,t_2=120,\beta=(0,0.5,-0.5)$ }
   \end{subfigure}\hfill
   \begin{subfigure}[b]{0.5\textwidth}   
      \centering \includegraphics[height=6.cm, width=6.cm,scale=0.3]{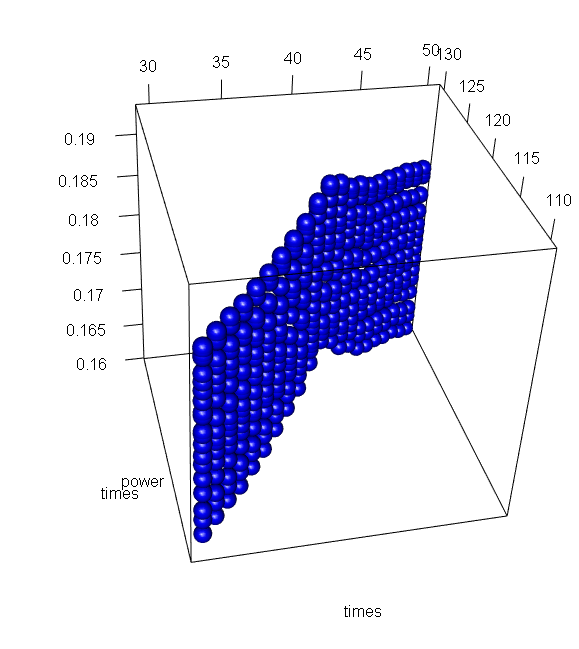}
      \caption{\large AR(1) model with \centering $t_1=40,t_2=120,\beta=(0,0.4,-0.4)$  }
   \end{subfigure}\hfill
\caption[]{Change-point locations for no change (a), one change (b) and two changes (c)-(d).}
\label{f3}
\end{figure}
\end{center}
\subsection{Concluding remarks} 
We have constructed and studied a likelihood-ratio test for weak changes detection in the mean of CHARN($p$) time series models. We have derived its power function and given its expression. We have used this power in some new strategies for detecting changes in the mean and estimating their locations. Simulation experiment has shown that our method performs well on the examples studied. Compared with some competitive methods, ours have shown a better performance in estimating the number and locations of weak changes. They have also been applied to three sets of real data studied in \cite{xiong}. Two changes have been detected in each of these data and their locations have been estimated. For each series, one of our location estimates is very close to the single one  obtained in \cite{xiong}. 

We did numerous trials from which we observed that, in almost all the cases, with strategy S1, no change was detected when the data were simulated with no change, a change was detected for data simulated with at least one change, and the location was accurately estimated for data containing one single change, even when the magnitude was too small.

%
Our methods are particularly adapted to weak change-point detection but also to {\it local} change-point detection. The latter means change-point detection in given areas of the data under study.
In using them as screening methods, changes that may have been previously detected by some other methods can be considered as known. So that, under the null hypothesis, the data may not be stationary. This constitutes a great advantage of our methods over existing ones in which the stationarity assumption under the null hypothesis plays a key role in the derivation of the theoretical results. 
%
%
%
\begin{center}
\begin{figure}[h!]
   \begin{subfigure}[b]{0.4\linewidth}
      \centering \includegraphics[scale=0.17]{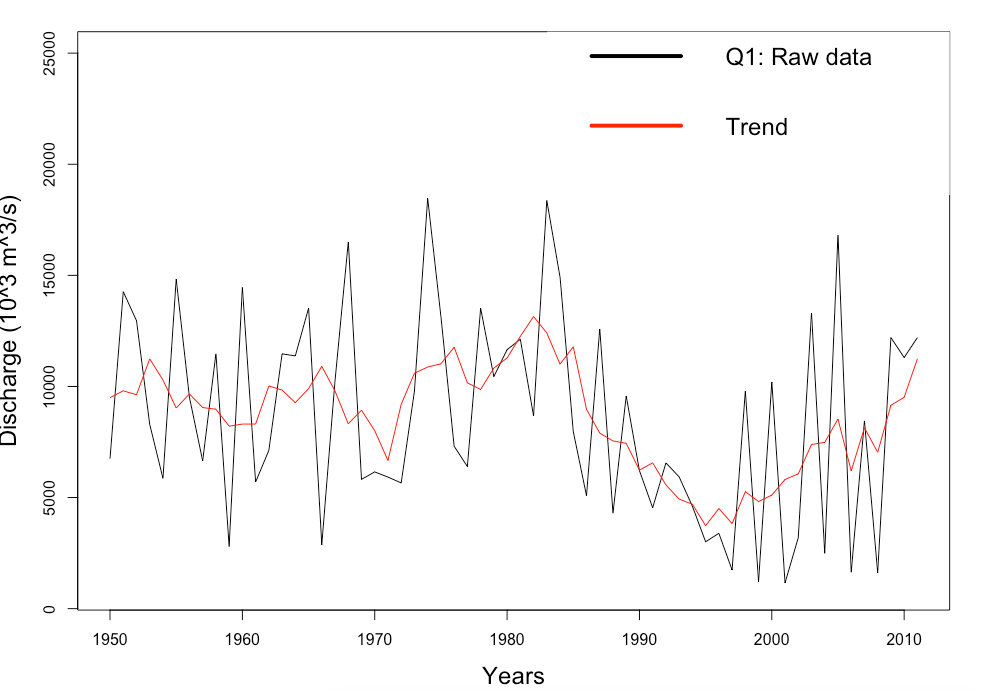}
      \caption{\centering $Q_1$: Raw data}
   \end{subfigure}\hfill
   \begin{subfigure}[b]{0.4\textwidth}   
      \centering \includegraphics[scale=0.17]{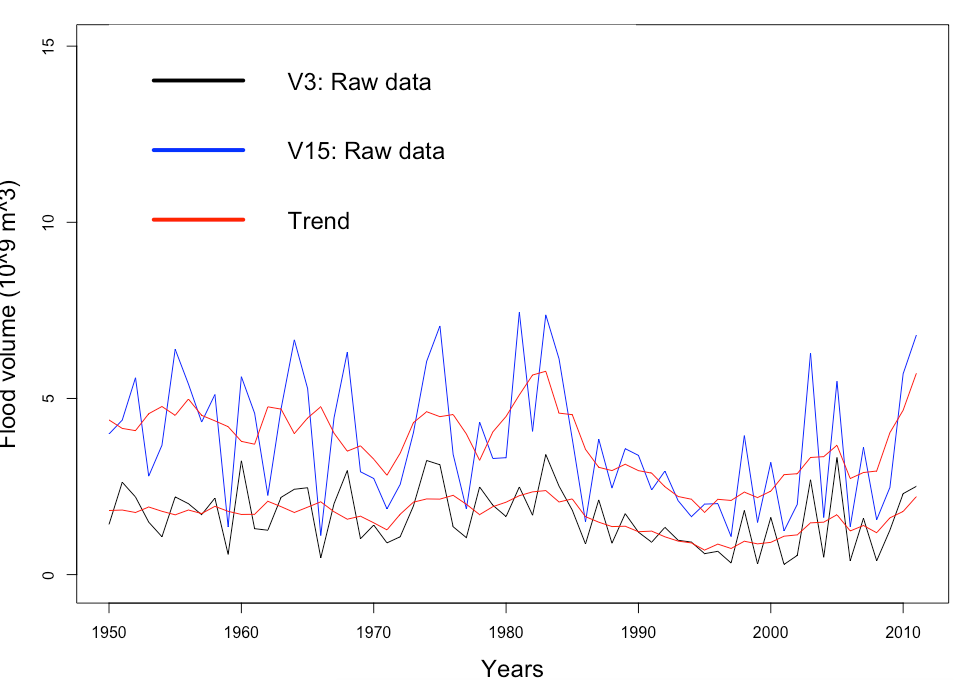}
      \caption{\centering  $V_3$ and $V_{15}$: Raw data}
   \end{subfigure}\hfill
   \begin{subfigure}[b]{0.4\linewidth}
      \centering \includegraphics[scale=0.17]{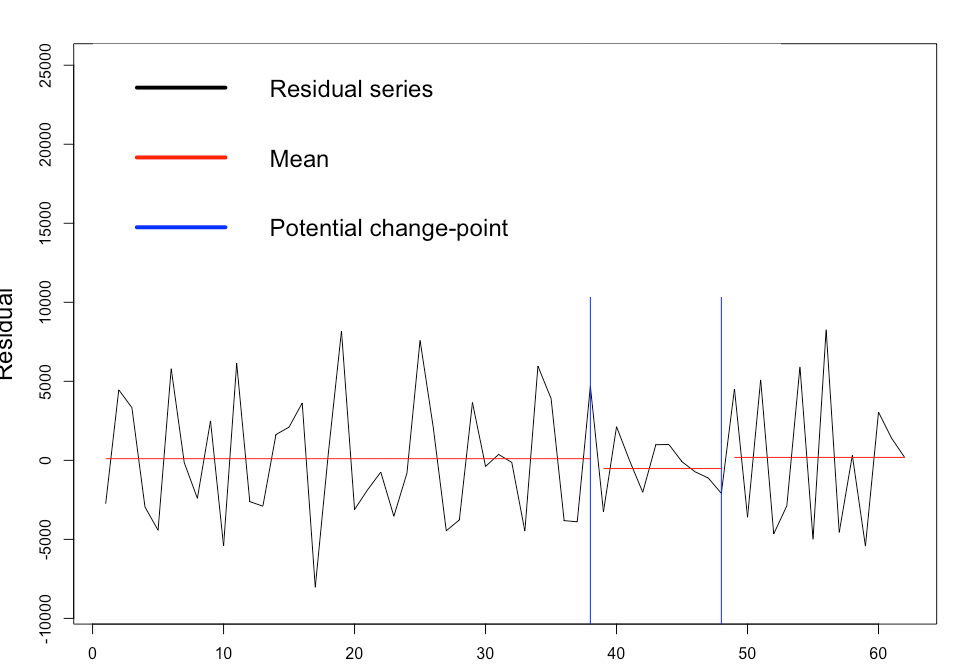}
      \caption{\centering Residual series}
      \end{subfigure} \hfill
      \begin{subfigure}[b]{0.4\textwidth}
      \centering \includegraphics[scale=0.17]{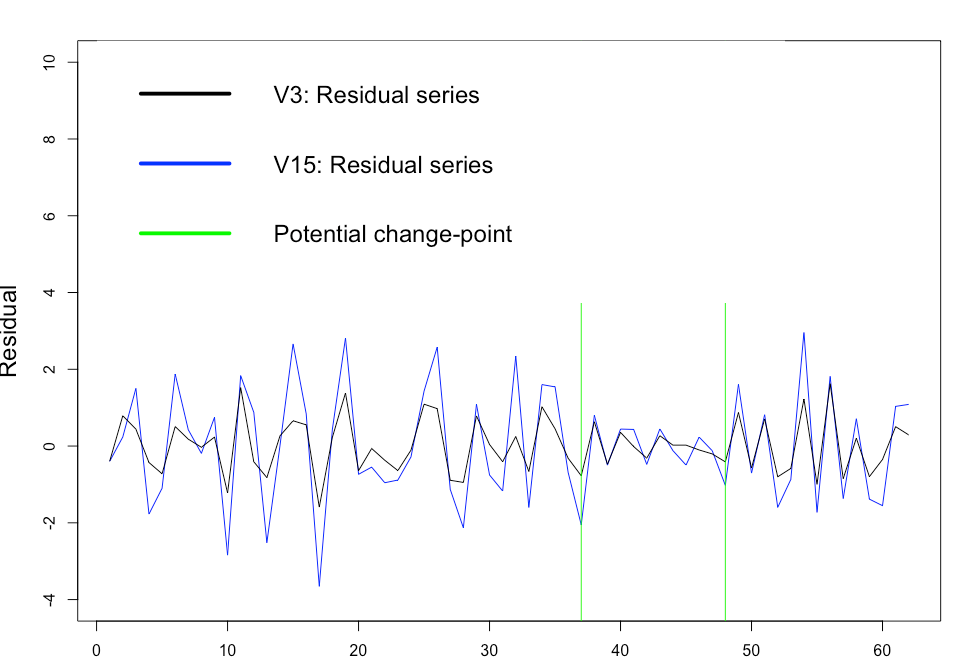}
      \caption{
       \centering Residual series}
      \end{subfigure}\hfill
      \begin{subfigure}[b]{0.4\linewidth}
      \centering \includegraphics[scale=0.17]{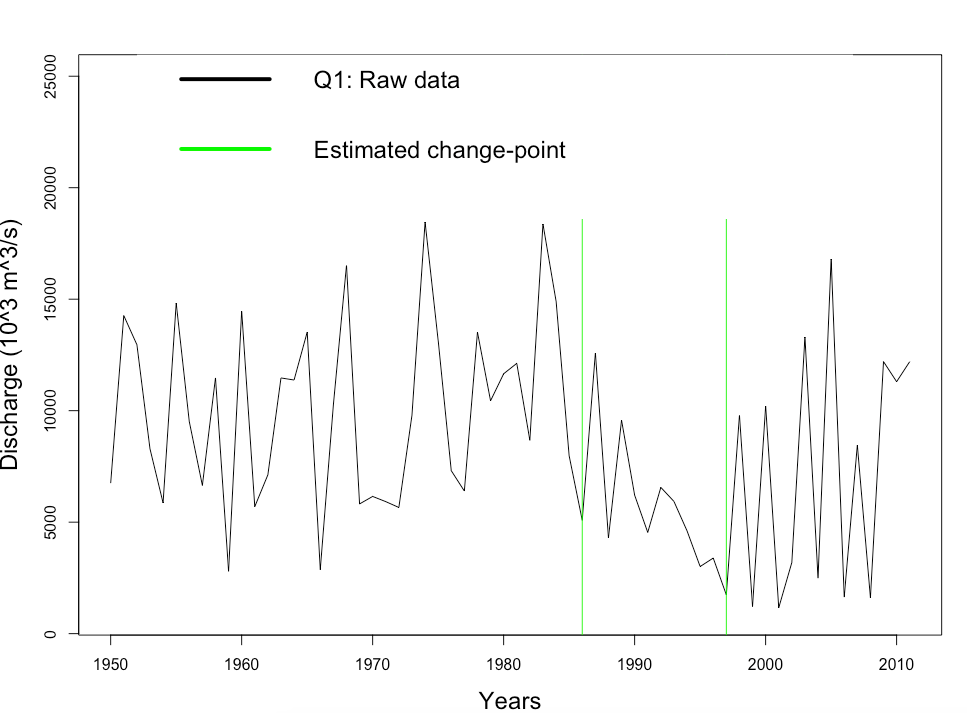}
      \caption{\centering Estimated change point}
      \end{subfigure} \hfill
      \begin{subfigure}[b]{0.4\textwidth}
      \centering \includegraphics[scale=0.17]{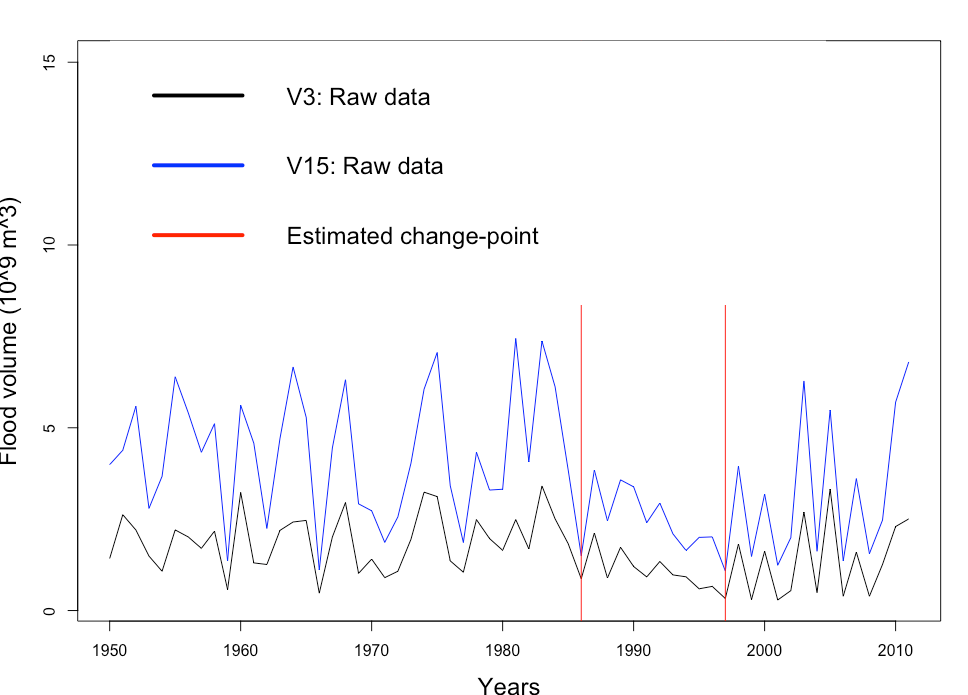}
      \caption{
       \centering Estimated change point}
      \end{subfigure}\hfill
   \caption[]{Change-point locations for real data.}
\label{f4}   
\end{figure}
\end{center}
\section{Appendix : Proofs}
This section provides the proofs of results.
\subsection{Proof of Theorem \ref{th1}}

\noindent 
For any $\beta \in \mathbb R^{k+1}$, 
the log-likelihood ratio of $ H_0 $ against $ H^{(n)}_{\beta}$ is given by
$$\Lambda_n(\gamma_0,\beta)=\sum_{t=1}^n{\left\lbrace\log\left[f(\varepsilon_t(\gamma_n))\right]-\log\left[f(\varepsilon_t(\gamma_0))\right]\right\rbrace}+o_P(1). $$
We first show that as $n\rightarrow\infty,$ $\Lambda_n(\gamma_0,\beta)$ decomposes into
 $$\Lambda_n(\gamma_0,\beta)=\Lambda_{1n}-\Lambda_{2n}+o_P(1),$$
where 
$$\Lambda_{1n}:= \dfrac{1}{\sqrt{n}}\sum_{t=1}^n{\dfrac{\beta^{\top}\omega(t)}{V(Z_{t-1})}\phi_f[\varepsilon_t(\gamma_0)]}$$
$$\Lambda_{2n}:=\dfrac{1}{2n}\sum_{t=1}^n{\dfrac{\left[\beta^{\top}\omega(t)\right]^2}{V^2(Z_{t-1})}}\phi'_f[\varepsilon_t({\gamma_0})].$$
%
By a first-order Taylor expansion of $\log f[\varepsilon_t(\gamma)]$ around $\gamma_0,$ one has for some $\widetilde{\gamma}$ between $\gamma_n$ and $\gamma_0$,
\begin{eqnarray*}
\Lambda_n(\gamma_0,\beta)&=&\left(\gamma_n-\gamma_0\right)^\top \sum_{t=1}^n{{\dfrac{\partial\left[\log f(\varepsilon_t(\gamma))\right]}{\partial{\gamma}}}\Big|_{\gamma=\gamma_0}}
+ \dfrac{1}{2} \left(\gamma_n-\gamma_0\right)^{\top}\dfrac{\partial^2\left[\log f(\varepsilon_t({\gamma}))\right]}{\partial{{\gamma}^2}}\Big|_{\gamma=\widetilde{\gamma}}\left(\gamma_n-\gamma_0\right).
\end{eqnarray*}
From the first- and second-order derivatives of $\log f[\varepsilon_t(\gamma)]$ will respect to $\gamma_0,$ one has 
\begin{eqnarray*}
\Lambda_n (\gamma_0,\beta)
&=& \dfrac{1}{\sqrt{n}}\sum_{t=1}^n{\dfrac{\beta^{\top}\omega(t)}{V(Z_{t-1})}\phi_f[\varepsilon_t(\gamma_0)]}- \dfrac{1}{2n}\sum_{t=1}^n{\dfrac{\left[\beta^{\top}\omega(t)\right]^2}{V^2(Z_{t-1})}}\phi'_f[\varepsilon_t(\widetilde{\gamma})]+o_P(1)\\
&=& \Lambda_{1n}-\Lambda^{(1)}_{1n}+o_P(1).
\end{eqnarray*}
Then we can write, by a first-order Taylor expansion of $\varepsilon_t(\gamma)$ around $\gamma_0,$ $$\Lambda^{(1)}_{1n}:=\dfrac{1}{2n}\sum_{t=1}^n{\dfrac{\left[\beta^{\top}\omega(t)\right]^2}{V^2(Z_{t-1})}}\phi'_f\left[\varepsilon_t(\gamma_0)+(\gamma_0-\widetilde{\gamma})^{\top}\dfrac{\omega(t)}{V(Z_{t-1})}\right].$$
From $(A_3)$, we have
$$ \left|\dfrac{1}{n}\sum_{t=1}^n{\dfrac{\left(\beta^{\top}\omega(t)\right)^2}{V^2(Z_{t-1})}}\phi'_f\left[\varepsilon_t(\gamma_0)+(\gamma_0-\widetilde{\gamma})^{\top}\dfrac{\omega(t)}{V(Z_{t-1})}\right]-\dfrac{1}{n}\sum_{t=1}^n{\dfrac{\left[\beta^{\top}\omega(t)\right]^2}{V^2(Z_{t-1})}}\phi'_f\left[\varepsilon_t(\gamma_0)\right]\right|$$ $$\leq \dfrac{1}{n} C_\phi \sum_{t=1}^n{\dfrac{\left[\beta^{\top}\omega(t)\right]^2}{V^2(Z_{t-1})}\Big|(\gamma_0-\widetilde{\gamma})^{\top}\dfrac{\omega(t)}{V(Z_{t-1})}\Big|}.$$
$$\leq (k+1) C_\phi ||\gamma_0-\widetilde{\gamma}|| \left(\dfrac{1}{n}\sum_{t=1}^n{\dfrac{\left[\beta^{\top}\omega(t)\right]^2}{V^3(Z_{t-1})}}\right).$$
Expanding the term in brackets in terms of a sum of sums over the $[t_{j-1},t_j)$'s, it is easy to see that, by the ergodic theorem it tends almost surely to $\sum_{j=1}^{k+1}\alpha_j \beta^2_j \mu_{j3}(\gamma_0)$.
Since $\widetilde{\gamma}$ is between $\gamma_0$ and $\gamma_n$ and $\gamma_n-\gamma_0={\beta/\sqrt{n}}$, it follows that $||\gamma_0-\widetilde{\gamma}||\leq {||\beta||/\sqrt{n}}$ which tends to zero as $ n $ goes to infinity. Whence almost surely,
$$\lim_{n\rightarrow \infty}{\dfrac{1}{n}  \sum_{t=1}^n{\dfrac{\left[\beta^{\top}\omega(t)\right]^2}{V^3(Z_{t-1})}||\gamma_0-\widetilde{\gamma}||}}=0.$$
Thus, for larger value of $n$,
$$\Lambda^{(1)}_{1n}=\Lambda_{2n}+o_P(1).$$
Let's show now that under $H_0$, as $n\rightarrow\infty,$
$$  {\Lambda_{1n}}\overset{D}{\longrightarrow} \mathcal{N}(0,\mu(\gamma_0,\beta)).$$
For this, we consider
$${\Lambda}_{n,j}=\dfrac{1}{\sqrt{n}}\sum_{t=1}^j{\dfrac{\beta^{\top}\omega(t)}{V(Z_{t-1})}\phi_f\left[\varepsilon_t(\gamma_0)\right]},\quad  j=1,\ldots n,$$
and we define for any $t=1, \ldots, n,$
$$\xi_{n,t}=\dfrac{1}{\sqrt{n}}{\dfrac{\beta^{\top}\omega(t)}{V(Z_{t-1})}\phi_f\left[\varepsilon_t(\gamma_0)\right]}.$$
Then we observe that $\Lambda_{n,j}$ is a zero-mean process that can be rewritten as 
$$\Lambda_{n,j}=\sum_{t=1}^j{\xi_{n,t}},\quad j=1,\ldots,n$$

\noindent First, we show  that $\lbrace({\Lambda}_{n,j},G_j),j=1,\ldots,n\rbrace$ is a sequence of martingales where we recall that $G_{t-1}=\sigma(Z_1,\ldots,Z_t)$ is a $\sigma$-algebra spanned  by $Z_1,\ldots,Z_t,$ $t\in\mathbb Z.$
\medskip

Now, let $j_1,j_2\in\mathbb Z$, $j_1<j_2$. We have:
\begin{eqnarray*}
 \mathbb E\left({\Lambda}_{n,j_2}\Big|G_{j_1}\right)
& =& \mathbb E\left[\dfrac{1}{\sqrt{n}}\sum_{t=1}^{j_1}{\dfrac{\beta^{\top}\omega(t)}{V(Z_{t-1})}\phi_f\left[\varepsilon_t(\gamma_0)\right]}\Big|G_{j_1}\right]
+  \mathbb E\left[\dfrac{1}{\sqrt{n}}\sum_{t={j_1}+1}^{j_2}{\dfrac{\beta^{\top}\omega(t)}{V(Z_{t-1})}\phi_f\left[\varepsilon_t(\gamma_0)\right]}\Big|G_{j_1}\right]\\
& = &\dfrac{1}{\sqrt{n}}\sum_{t=1}^{j_1}{\dfrac{\beta^{\top}\omega(t)}{V(Z_{t-1})}\phi_f\left[\varepsilon_t(\gamma_0)\right]}
+\dfrac{1}{\sqrt{n}}\sum_{t={j_1}+1}^{j_2}{\mathbb E\left[\dfrac{\beta^{\top}\omega(t)}{V(Z_{t-1})}\Big|G_{j_1}\right] \mathbb E\left\lbrace\phi_f\left[\varepsilon_t(\gamma_0)\right]\right\rbrace.}
\end{eqnarray*}
From Remark \ref{rm2}, we have $\mathbb E\left\lbrace \phi_f\left[\varepsilon_{t}(\gamma_0)\right]\right\rbrace=0$.
Then
 $$ \mathbb E\left({\Lambda}_{n,j_2}\Big|G_{j_1}\right)=\Lambda_{n,j_1}.$$
Also, we have
\begin{eqnarray*}
\sum_{t=1}^n{\mathbb E(\xi^2_{n,t}|G_{t-1} )} 
&{=}& \dfrac{1}{n}\sum_{t=1}^n{\dfrac{\left[\beta^{\top}\omega(t)\right]^2}{V^2(Z_{t-1})} \mathbb E\left\lbrace\phi_f^2\left[\varepsilon_t(\gamma_0)\right]\Big|G_{t-1}\right\rbrace}.
\end{eqnarray*}
Since for any $t=1,\ldots,n,$ $\varepsilon_t$ is independent of $G_{t-1},$ we have 
\begin{eqnarray*}
\sum_{t=1}^n{\mathbb E(\xi^2_{n,t}|G_{t-1} )}
&=& \dfrac{1}{n}\sum_{j=1}^{k+1}{\sum_{t={t_{j-1}}}^{t_j}{\dfrac{\beta^2_j}{V^2(Z_{t-1})} \mathbb E\left\lbrace\phi_f^2\left[\varepsilon_t(\gamma_0)\right]\right\rbrace}}\\
&=&  \sum_{j=1}^{k+1}{ \dfrac{n_j(n)}{n}\dfrac{\beta^2_j}{n_j(n)}\sum_{t={t_{j-1}}}^{t_j}{\dfrac{1}{V^2(Z_{t-1})} \mathbb E\left\lbrace\phi_f^2\left[\varepsilon_t(\gamma_0)\right]\right\rbrace}}.
\end{eqnarray*}
From our iid assumption on the $\varepsilon_t'$s we have $$\mathbb E\left\lbrace\phi_f^2\left[\varepsilon_t(\gamma_0)\right]\right\rbrace=\displaystyle\int_{\mathbb R}{\phi^2_f(x)f(x)}dx=I(f).$$
From the ergodic theorem, for any $j=1,\ldots,k+1,$ almost surely,
$$\lim_{n\rightarrow\infty}{\dfrac{1}{n_j(n)}\sum_{t={t_{j-1}}}^{t_j}{\dfrac{1}{V^2(Z_{t-1})}\mathbb E\left\lbrace\phi_f^2\left[\varepsilon_t(\gamma_0)\right]\right\rbrace}}=\mu_{j2}<\infty.$$
Then, 
$$\lim_{n\rightarrow\infty}{\sum_{t=1}^n{\mathbb E(\xi^2_{n,t}|G_{t-1} )}}=\mu(\gamma_0,\beta)=\sum_{j=1}^{k+1}{\alpha_j\beta^2_j\mu_{j2}}(\gamma_0)<\infty.$$
\noindent Now, we check the conditional Lindeberg conditions. 
\medskip

\noindent
Let $\epsilon>0$,
using the H\"older conditional inequality, we have
\begin{eqnarray*}
&&\sum_{t=1}^n\mathbb E\left(\xi^2_{n,t}\nbI_{|\xi_{n,t}|>\epsilon }| G_{t-1}\right)
\le
\sum_{t=1}^n \mathbb E^{\frac{2}{3}}\left( \xi^3_{n,t} |G_{t-1}\right) \mathbb E^{\frac{1}{3} }\left(\nbI_{|\xi_{n,t}|>\epsilon} | G_{t-1}\right)\\
& = & \sum_{t=1}^n\mathbb E^{\frac{2}{3}} \left[ n^{-\frac{3}{2}}\dfrac{\left[\beta^{\top}\omega(t)\right]^3}{V^3(Z_{t-1})} \phi^3_f\left[\varepsilon_t(\gamma_0)\right]\Big|G_{t-1}\right]
 \mathbb P^{\frac{1}{3}}\left(|\xi_{n,t}|>\epsilon\Big|G_{t-1}\right).
 \end{eqnarray*}
From Markov inequality, we have  
 \begin{eqnarray*}
&& \sum_{t=1}^n{\mathbb E\left(\xi^2_{n,t}\nbI_{|\xi_{n,t}|>\epsilon }| G_{t-1}\right)}\\
&\leq & Cst\dfrac{1}{n} \sum_{t=1}^n{\dfrac{\left[\beta^{\top}\omega(t)\right]^2}{V^2(Z_{t-1})} \mathbb E^{\frac{2}{3}}\left\lbrace\phi_f^3\left[\varepsilon_t(\gamma_0)\right]\Big|G_{t-1}\right\rbrace} 
\times  \mathbb E^{\frac{1}{3}}\left[ n^{-\frac{3}{2}}\dfrac{|\beta^{\top}\omega(t)|^3}{V^3(Z_{t-1})}|\phi_f[\varepsilon_t(\gamma_0)]|^3 \Big| G_{t-1}\right] \notag \\
  &\leq & Cst \sum_{t=1}^n{n^{-\frac{3}{2}}\dfrac{|\beta^{\top}\omega(t)|^3}{V^3(Z_{t-1})}\mathbb E\left\lbrace |\phi_f\left[\varepsilon_t(\gamma_0)\right]|^3 \right\rbrace} \notag \\
 &=& Cst \sum_{j=1}^{k+1}{\left(\dfrac{n_j(n)}{n}\right)^{\frac{3}{2}}\dfrac{|\beta_j|^3}{(n_j( n))^{\frac{3}{2}}}\sum_{t=t_{j-1}}^{t_j}{\dfrac{1}{V^3(Z_{t-1})}\mathbb E\left\lbrace|\phi_f\left[\varepsilon_t(\gamma_0)\right]|^3 \right\rbrace}}.
\end{eqnarray*}
For any $j=1,\ldots,k+1$, observing that 
\begin{eqnarray*}
{n_j^{-\frac{3}{2}}(n)\sum_{t=t_{j-1}}^{t_j}{\dfrac{1}{V^3(Z_{t-1})}\mathbb E\left\lbrace|\phi_f\left[\varepsilon_t(\gamma_0)\right]|^3 \right\rbrace}}
&=& {n_j^{-\frac{1}{2}}(n)}
\times { n_j^{-1}(n)\sum_{t=t_{j-1}}^{t_j}{\dfrac{1}{V^3(Z_{t-1})}\mathbb E\left\lbrace|\phi_f\left[\varepsilon_t(\gamma_0)\right]|^3 \right\rbrace}}
\end{eqnarray*}
and noting that the second term in the right-hand side of the above equation tends almost surely to a finite limit, we conclude easily that for all positive $\epsilon$, 
$$\lim_{n\rightarrow\infty}{\sum_{t=1}^n{\mathbb E\left(\xi^2_{n,t}\nbI_{|\xi_{n,t}|>\epsilon }| G_{t-1}\right)}}=0\quad a.s.$$
Thus, the conditions of Corollary 3.1 of \cite{hall} are verified.
Accordingly, under $H_0,$
$$
{\Lambda_{1n}}\overset{D}{\longrightarrow} \mathcal{N}(0,\mu(\gamma_0,\beta)), \ n\rightarrow\infty.$$
Now, turning to the study of $\Lambda_{2n}$, it is easy to see that 
$$\Lambda_{2n}= \dfrac{1}{2} \sum_{j=1}^{k+1}{\dfrac{n_j(n)}{n}\dfrac{\beta^2_j}{n_j(n)} \sum_{t=t_{j-1}}^{t_j}{\dfrac{1}{V^2(Z_{t-1})}}\phi'_f[\varepsilon_t({\gamma_0})]}.$$
Since, under our assumptions, $\mathbb E\lbrace\phi'_f[\varepsilon_1({\gamma_0})]\rbrace=I(f),$ by Remark $(ii)$ and the ergodic theorem, for all $ j = 1, \ldots, k+1, $ we have, almost surely
 $$\lim_{n\rightarrow +\infty}{\dfrac{1}{n_j(n)} \sum_{t=t_{j-1}}^{t_j}{\dfrac{1}{V^2(Z_{t-1})}\phi'_f[\varepsilon_t({\gamma_0})]}}=I(f){\displaystyle\int_{\mathbb R^p}{\dfrac{1}{V^2(x)}}dF_j(x)}=\mu_{j2}(\gamma_0).$$
Thus, almost surely,
$$\lim_{n\rightarrow +\infty}{\Lambda_{2n}}=\dfrac{1}{2}\sum_{j=1}^{k+1}{\alpha_j \beta^2_j \mu_{j2}}(\gamma_0)=\dfrac{\mu(\gamma_0,\beta)}{2},$$
and we can conclude that, as $n\rightarrow\infty,$
$$\Lambda_n(\gamma_0,\beta)=\Lambda_{1n}-\dfrac{\mu(\gamma_0,\beta)}{2}+o_P(1).$$
It results from the fact that under $H_0$, $\Lambda_{1n}\overset{D}{\longrightarrow} \mathcal{N}(0,\mu(\gamma_0,\beta)),$ as $n\rightarrow\infty,$ that under $H_0$, as $n\rightarrow\infty,$
$$\Lambda_n(\gamma_0,\beta)\overset{D}{\longrightarrow} \mathcal{N}\left(-\dfrac{\mu(\gamma_0,\beta)}{2},\mu(\gamma_0,\beta)\right).$$ 
Finally, the LAN property is established for the central sequence $\Delta_n(\gamma_0,\beta)\equiv\Lambda_{1n}.$
\subsection{Proof of Theorem \ref{th2}}
\noindent It follows immediately from Theorem \ref{th1} that, under $H_0$, as $n\rightarrow\infty,$
$$\begin{pmatrix} 
\Delta_n(\gamma_0,\beta) \\ 
\Lambda_n(\gamma_0,\beta)
 \end{pmatrix}\overset{D}{\longrightarrow} \mathcal{N}\left(\begin{pmatrix} 
0 \\ 
-\dfrac{\mu(\gamma_0,\beta)}{2}
 \end{pmatrix} , \begin{pmatrix} 
\mu(\gamma_0,\beta) &\mu (\gamma_0,\beta)\\ 
\mu(\gamma_0,\beta) & \mu(\gamma_0,\beta)
 \end{pmatrix} \right).$$
 The part (i) is an easy consequence of Theorem \ref{th1}.
\smallskip

\noindent
For the part (ii), from \cite{dreosebeke} (see Proposition 4.2), under $H^{(n)}_{\beta},$ as $n\rightarrow\infty,$
$$\Delta_n(\gamma_0,\beta)\overset{D}{\longrightarrow} \mathcal{N}(\mu(\gamma_0,\beta),\mu(\gamma_0,\beta)).$$
Observing that for all $n\geq 1,$ $$\dfrac{\Delta_n(\gamma_0,\beta)}{\widehat{\varpi}_n(\gamma_0,\beta)}=\dfrac{\Delta_n(\gamma_0,\beta)}{\varpi(\gamma_0,\beta)}\times \dfrac{\varpi(\gamma_0,\beta)}{\widehat{\varpi}_n(\gamma_0,\beta)}.$$
Since, under $H_0,$ as $n\rightarrow \infty,$ $\widehat{\varpi}_n(\gamma_0,\beta)\longrightarrow  \varpi(\gamma_0,\beta),$ by contiguity, the result still holds under $H^{(n)}_{\beta}.$ It follows from Theorem \ref{th1} that, under $H^{(n)}_{\beta},$ as $n\rightarrow \infty,$ $$\dfrac{\Delta_n(\gamma_0,\beta)}{\varpi(\gamma_0,\beta)}\overset{D}{\longrightarrow} \mathcal{N}(\varpi(\gamma_0,\beta),1).$$
Accordingly, we obtain the asymptotic cumulative distribution of $\mathcal{T}_n(\gamma_0,\beta)$ under $H^{(n)}_{\beta}$ given by 
\begin{eqnarray*}
 \lim_{n\rightarrow\infty}{\mathbb P\left(\dfrac{\Delta_n(\gamma_0,\beta)}{\widehat{\varpi}_n(\gamma_0,\beta)}>u_\alpha|H^{(n)}_{\beta}\right)}&=&\lim_{n\rightarrow\infty}{\mathbb P\left(\dfrac{\Delta_n(\gamma_0,\beta)}{{\varpi}(\gamma_0,\beta)}>u_\alpha|H^{(n)}_{\beta}\right)}\\
&=&1-\Phi\left(u_\alpha-\varpi (\gamma_0,\beta)\right),
\end{eqnarray*}
where for $\alpha \in (0,1),$ $u_\alpha$ is the   ($1-\alpha$)-quantile of a standard normal distribution with cumulative distribution $\Phi.$
\smallskip

\noindent
The part (iii) results from the LAN property and the fact that the limiting model is a Gaussian location model (see, eg, \cite{lecam} or \cite{dreosebeke}).
\subsection{Proof of Proposition \ref{pr1}}
(i)
Let for any $1\leq j \leq n,$
$$L_{n,j}=n^{-\frac{1}{2}}\sum_{t=1}^j{\left(\Psi_1(Z_{t-1},\psi_0)\Omega_1[\varepsilon_t(\psi_0,\gamma_0)],\Psi_2(Z_{t-1},\psi_0)\Omega_2[\varepsilon_t(\psi_0,\gamma_0)]\right)}+o_P(1),$$
and $\eta=(\eta^{\top}_1,\eta^{\top}_2)^{\top}\in\mathbb R^{l+q}$ with $\eta_1=(\eta^1_1,\ldots, \eta^l_1)^{\top}$ and $\eta_2=(\eta^1_2,\ldots, \eta^q_2)^{\top}$.\\
We observe that for any $n\geq1$ and $1\leq j \leq n,$  
$$\eta^{\top} L_{n,j}=n^{-\frac{1}{2}}\sum_{t=1}^j{a_t(\eta)},$$
$$a_t(\eta)=\sum_{\ell=1}^2{\eta^{\top}_\ell \Psi_\ell(Z_{t-1},\psi_0)\Omega_\ell[\varepsilon_t(\psi_0,\gamma_0)]}, \ t\in\mathbb Z.$$
It's easy to see that for any $j=1,\ldots,n,$ $\eta^{\top}{L}_{n,j}$ is centred at 0.\\
\noindent We first show that $\lbrace(\eta^{\top}{L}_{n,j},G_j),j=1,\ldots,n\rbrace$ is a sequence of martingales. Since the proof is very similar to that of $\Lambda_{n,j}$ in the proof of Theorem \ref{th1}, we don't detail it much.
\smallskip

\noindent
Let $j_1,j_2\in\mathbb Z$ such that $j_1<j_2$. We have:
\begin{eqnarray*}
\mathbb E(\eta^{\top}{L}_{n,j_2}|G_{j_1})&=&\mathbb E(\eta^{\top}{L}_{n,j_1}|G_{j_1})+n^{-\frac{1}{2}} \sum_{t={j_1+1}}^{j_2}{\mathbb E[a_t(\eta)|G_{j_1}]}\\
&=& \eta^{\top}{L}_{n,j_1}+n^{-\frac{1}{2}} \sum_{t={j_1+1}}^{j_2}{\sum_{\ell=1}^2{\mathbb E[\eta^{\top}\Psi_\ell(Z_{t-1},\psi_0)]\mathbb E\lbrace\Omega_\ell[\varepsilon_t(\psi_0,\gamma_0)]\rbrace}}.
\end{eqnarray*}
As for any $\ell=1,2$ and $t\in\mathbb Z,$ $\mathbb E\lbrace\Omega_\ell[\varepsilon_t(\psi_0,\gamma_0)]\rbrace=0$, 
we have $\mathbb E(\eta'{L}_{n,j_2}|G_{j_1})= \eta'{L}_{n,j_1}.$
\smallskip

\noindent
Also,
\begin{equation}
\sum_{t=1}^n{\mathbb E\lbrace[n^{-\frac{1}{2}}a_t(\eta)]^2|G_{t-1} \rbrace}= \dfrac{1}{n}\sum_{t=1}^n{\mathbb E[a^2_t(\eta)|G_{t-1}]}.
\end{equation}
Replacing $ a^2_t (\eta) $ by its expression, we get
\begin{eqnarray*}
&&\sum_{t=1}^n{\mathbb E\lbrace[n^{-\frac{1}{2}}a_t(\eta)]^2|G_{t-1} \rbrace}
\\
&=& \dfrac{1}{n}\sum_{t=1}^n{\mathbb E\left(\left\lbrace\eta^{\top}_1 \Psi_1(Z_{t-1},\psi_0)\Omega_1[\varepsilon_t(\psi_0,\gamma_0)]+\eta^{\top}_2\Psi_2(Z_{t-1},\psi_0)\Omega_2[\varepsilon_t(\psi_0,\gamma_0)]\right \rbrace^2  \Big|G_{t-1}\right)}
\end{eqnarray*}
Using the fact that for any $ t \in \mathbb Z,$ $ \varepsilon_t $ independent of $ G_ {t-1} $, it is easy to see that 
\begin{eqnarray*}
& & 
\sum_{t=1}^n{\mathbb E\lbrace[n^{-\frac{1}{2}}a_t(\eta)]^2|G_{t-1} \rbrace}
=\sum_{j=1}^{k+1}{\dfrac{n_j(n)}{n}\dfrac{1}{n_j(n)}\sum_{t={t_{j-1}}}^{t_j}{ [\eta^{\top}_1  \Psi_1(Z_{t-1},\psi_0)]^2\mathbb E\Big\lbrace\Omega^2_1[\varepsilon_t(\psi_0,\gamma_0)]\Big\rbrace}}\\
&+& \sum_{j=1}^{k+1}{\dfrac{n_j(n)}{n}\dfrac{1}{n_j(n)}\sum_{t={t_{j-1}}}^{t_j}{ [\eta^{\top}_2  \Psi_2(Z_{t-1},\psi_0)]^2 \mathbb E\Big\lbrace\Omega^2_2[\varepsilon_t(\psi_0,\gamma_0]\Big\rbrace}}\\
 &+&\sum_{j=1}^{k+1}{\dfrac{n_j(n)}{n}\dfrac{1}{n_j(n)}\sum_{t={t_{j-1}}}^{t_j}{\eta^{\top}_1\Psi_1(Z_{t-1},\psi_0).\eta^{\top}_2\Psi_2(Z_{t-1},\psi_0)\mathbb E\left\lbrace\Omega_1[\varepsilon_t(\psi_0,\gamma_0)] \Omega_2[\varepsilon_t(\psi_0,\gamma_0)]\right\rbrace}}\\
 &=& S^{(1)}_n+S^{(2)}_n+S^{(3)}_n.
\end{eqnarray*}
From the ergodic theorem, we have, as $n\rightarrow\infty,$
 $$S^{(1)}_n\overset{a.s}{\longrightarrow}\sum_{j=1}^{k+1}{\alpha_j\int_{\mathbb R^p}{{[\eta^{\top}_1} \Psi_1(x,\psi_0)]^2}dF_j(x)\int_{\mathbb R}{ \Omega^2_1(x)f(x)}dx}<\infty,$$
$$S^{(2)}_n\overset{a.s}{\longrightarrow}\sum_{j=1}^{k+1}{\alpha_j\int_{\mathbb R^p}{[{\eta^{\top}_2} \Psi_2(x,\psi_0)]^2}dF_j (x)\int_{\mathbb R}{\Omega^2_2(x)f(x)}dx}<\infty,$$
$$S^{(3)}_n\overset{a.s}{\longrightarrow}2\sum_{j=1}^{k+1}{\alpha_j\int_{\mathbb R^p}{\eta^{\top}_1\Psi_1(x,\psi_0).\eta^{\top}_2\Psi_2(x,\psi_0)}dF_j(x)\int_{\mathbb R}{\Omega_1(x)\Omega_2(x)f(x)}dx}<\infty.$$
Whence, as $n\rightarrow\infty,$ almost surely, 
$$\sum_{t=1}^n{E[(n^{-\frac{1}{2}}a_t(\eta))^2|G_{t-1} ]}\overset{}{\longrightarrow} \xi$$
\begin{eqnarray*}
\xi &=& \sum_{j=1}^{k+1}{\alpha_j}\Big[ \int_{\mathbb R}{ \Omega^2_1(x)f(x)}dx\int_{\mathbb R^p}{[{\eta^{\top}_2} \Psi_1(x,\psi_0)]^2}dF_j(x)
+ \int_{\mathbb R}{\Omega^2_2(x)f(x)}dx\int_{\mathbb R^p}{[{\eta^{\top}_1}\Psi_2(x,\psi_0)]^2}dF_j(x) \\
&&+  \int_{\mathbb R}{\Omega_1(x)\Omega_2(x)f(x)}dx\int_{\mathbb R^p}{\eta^{\top}_1\Psi_1(x,\psi_0)\eta^{\top}_2\Psi_2(x,\psi_0)}dF_j(x)\Big].
 \end{eqnarray*}
%
It remains to check the Lindeberg condition. This can be done in the same line as in the proof of Theorem \ref{th1}.
%
%
Hence, we deduce from Corollary 3.1 of \cite{hall} that, under $H_0,$ 
$$\eta^{\top}\sqrt{n}(\psi_n-\psi_0)\overset{D}{\longrightarrow}\mathcal{N}(0,\xi), \ n\rightarrow\infty,$$
from which, it results that, under $H_0,$ as $n\rightarrow\infty,$
$$\sqrt{n}(\psi_n-\psi_0)\overset{D}{\longrightarrow}\mathcal{N}(\mathbf{0},\Sigma)$$
$$\Sigma=\begin{pmatrix} 
\Sigma_{11}& \Sigma_{12} \\ 
\Sigma_{21}  & \Sigma_{22}
\end{pmatrix}$$
with $\Sigma_{11},$ $\Sigma_{12},$ $\Sigma_{21}$ and $\Sigma_{22}$ defined in the proposition.
\smallskip

\noindent 
(ii)
Under $H_0,$ we have by the part (i) and Theorem \ref{th1}, that as $n\rightarrow\infty,$
$$\sqrt{n}(\psi_n-\psi_0)\overset{D}{\longrightarrow}\mathcal{N}(\mathbf{0},\Sigma),$$
$$\Lambda_n(\psi_0,\gamma_0,\beta)\overset{D}{\longrightarrow}\mathcal{N}\left(-\dfrac{1}{2}\mu(\psi_0,\gamma_0,\beta),\mu(\psi_0,\gamma_0,\beta)\right),$$
where we recall that 
$$\mu(\psi_0,\gamma_0,\beta)=\sum_{j=1}^{k+1}{\alpha_j \beta^2_j \mu_{j2}(\psi_0,\gamma_0)}, $$
with 
$\mu_{j2}(\psi_0,\gamma_0)=I(f){\int_{\mathbb R^p}{\dfrac{1}{V_{\theta_0}^2(x)}}dF_j(x)},  j=1,\ldots,k+1$.
\smallskip

\noindent
Letting $S_n=\sqrt{n}(\psi_n-\psi_0)$, it results that under $H_0,$ $(S^{\top}_n,\Lambda_n(\psi_0,\gamma_0,\beta))^{\top}$ follows asymptotically a normal distribution with expectation $\left(\textbf{0}^{\top},-\dfrac{1}{2}\mu(\psi_0,\gamma_0,\beta)\right)^{\top}$ and covariance vector  
\begin{eqnarray*}
\lim_{n\rightarrow\infty}{\mathbb C\mbox{ov}(S_n,\Lambda_n(\psi_0,\gamma_0,\beta))}&=&\lim_{n\rightarrow\infty}{\mathbb C\mbox{ov}\left(S_n,\Delta_n(\psi_0,\gamma_0,\beta)-\dfrac{1}{2}\mu(\psi_0,\gamma_0,\beta)\right)}\\
&=& \lim_{n\rightarrow\infty}{\mathbb C\mbox{ov}\left(S_n,\Delta_n(\psi_0,\gamma_0,\beta)\right)}.
\end{eqnarray*}
Using the expression of $ \Delta_n (\psi_0, \gamma_0) $, we get 
\begin{eqnarray*}
\mathbb C\mbox{ov}(S_n,\Lambda_n(\psi_0,\gamma_0,\beta)&=&\dfrac{1}{\sqrt{n}}\sum_{t=1}^n{\beta^{\top}\omega(t) \mathbb E\left\lbrace S_n \dfrac{\phi_f[\varepsilon_t(\psi_0,\gamma_0)]}{V_{\theta_0}(Z_{t-1})}\right\rbrace}
\\&-&
\dfrac{1}{\sqrt{n}} \sum_{t=1}^n{\beta^{\top}\omega(t) \mathbb E(S_n) \mathbb E\left\lbrace\dfrac{\phi_f[\varepsilon_t(\psi_0,\gamma_0)]}{V_{\theta_0}(Z_{t-1})}\right\rbrace}\\
&=& C_{1n}+C_{2n}.
\end{eqnarray*}
Since $\mathbb E\lbrace \phi_f[\varepsilon_t(\psi_0,\gamma_0)]\rbrace=0,$ $t\in\mathbb Z,$ it follows that $C_{2n}= \mathbf{0}$, $\forall n.$ 
\smallskip

\noindent
Now,
\begin{eqnarray*}
C_{1n}&=& \dfrac{1}{{n}}\sum_{t=1}^n{\beta^{\top}\omega(t)}\mathbb E\left\{\Psi (Z_{t-1},\psi_0) \Omega^{\top}[\varepsilon_t(\psi_0,\gamma_0)]\dfrac{\phi_f[\varepsilon_t(\psi_0,\gamma_0)]}{V_{\theta_0}(Z_{t-1})}\right\}\\
&+& \dfrac{1}{{n}}\sum_{s,t=1,s\ne t}^n{\beta^{\top}\omega(t)}\mathbb E\left\{\Psi (Z_{s-1},\psi_0) \Omega^{\top}[\varepsilon_s(\psi_0,\gamma_0)]\dfrac{\phi_f[\varepsilon_t(\psi_0,\gamma_0)]}{V_{\theta_0}(Z_{t-1})}\right\}\\
&=& C_{11n}+C_{12n}.
\end{eqnarray*}
Since $\mathbb E\lbrace\Omega_1[\varepsilon_t(\psi_0,\gamma_0)]\rbrace=\mathbb E\lbrace\Omega_2[\varepsilon_t(\psi_0,\gamma_0)]\rbrace=\mathbb E\lbrace\phi_f[\varepsilon_t(\psi_0,\gamma_0)]\rbrace)=0$ for all $t\in\mathbb Z,$ and since $(\varepsilon_t)_{t\in\mathbb Z}$ is iid and for any $t\in\mathbb Z,$  $\varepsilon_t$ is independent of $G_{t-1}=\sigma(Z_1,\ldots,Z_t)$, it follows that $C_{12n}=\mathbf{0}.$
\smallskip

\noindent
For the study of $C_{11n},$ we have 
\begin{eqnarray*}
C_{11n}&=& \int_{\mathbb R}{ \Omega^{\top}(x)\phi_f(x)f(x)}dx\dfrac{1}{{n}}\sum_{t=1}^n{\beta^{\top}\omega(t)}\mathbb E\left[ \dfrac{\Psi (Z_{t-1},\psi_0)}{V_{\theta_0}(Z_{t-1})}\right]\\
&=& \int_{\mathbb R}{ \Omega^{\top}(x)\phi_f(x)f(x)}dx}\sum_{j=1}^{k+1}{\dfrac{n_j(n)}{n}\dfrac{\beta_j}{n_j(n)}\sum_{t=t_{j-1}}^{t_j}{\mathbb E\left[ \dfrac{\Psi (Z_{t-1},\psi_0)}{V_{\theta_0}(Z_{t-1})}\right] }.
\end{eqnarray*}
Letting $n\rightarrow\infty,$ it is easy to see that the right member of the above equality tends almost surely to $$
\int_{\mathbb R}{ \Omega^{\top}(x)\phi_f(x)f(x)}dx\sum_{j=1}^{k+1}{\alpha_j \beta_j}\int_{\mathbb R^p}{\dfrac{\Psi(x,\psi_0)}{V_{\theta_0}(x)}}dF_j(x).$$
Then, under $H_0,$ as $n\rightarrow\infty,$
$$\mathbb C\mbox{ov}(S_n,\Lambda_n(\psi_0,\gamma_0,\beta))\overset{}{\longrightarrow} \nu=\int_{\mathbb R}{ \Omega^{\top}(x)\phi_f(x)f(x)}dx\sum_{j=1}^{k+1}{\alpha_j \beta_j}\displaystyle\int_{\mathbb R^p}{\dfrac{\Psi(x,\psi_0)}{V_{\theta_0}(x)}}dF_j(x).$$
Hence, under $H_0$, as $n\rightarrow\infty,$
$$\begin{pmatrix} 
S_n \\ 
\Lambda_n(\psi_0,\gamma_0,\beta)
\end{pmatrix}\overset{D}{\longrightarrow} \mathcal{N}\left( \begin{pmatrix} 
\mathbf{0}\\ 
-\dfrac{1}{2}\mu(\psi_0,\gamma_0,\beta)
\end{pmatrix}, \begin{pmatrix} 
\Sigma & \nu^{\top} \\ 
\nu & \mu(\psi_0,\gamma_0,\beta)
\end{pmatrix} \right).$$
By Le Cam's third lemma, under $H_1^{(n)}$,  as $n\rightarrow\infty,$
$$S_n=\sqrt{n}(\psi_n-\psi_0)\overset{D}{\longrightarrow}\mathcal{N}(\nu,\Sigma).$$
\subsection{Proof of Proposition \ref{pr2}}
For any $\psi=(\rho^{\top},\theta^{\top})^{\top}\in \Theta \times \widetilde{\Theta}$ and $(\gamma,\beta)\in\mathbb R^{k+1}\times \mathbb R^{k+1},$ define
$$ \Lambda_n(\psi,\gamma,\beta)=\sum_{t=1}^n{\log\left\lbrace f\left[\varepsilon_t\left(\psi,\gamma+\dfrac{\beta}{\sqrt{n}}\right)\right]\right\rbrace-\log\left\lbrace f\varepsilon_t\left[(\psi,\gamma)\right]\right\rbrace}+o_P(1).$$
Then, the log-likelihood ratio of $H_0$ against $H^{(n)}_{\beta}$ is $\Lambda_n(\psi_0,\gamma_0,\beta).$
%
%
Writing a first-order Taylor expansion of $\Delta_n(.,\gamma_0,\beta)$ around $\psi_n$, we have, for some $\widetilde{\psi}_n$ lying between $\psi_n$ and $\psi_0.$
\begin{eqnarray} 
\Delta_n(\psi_0,\gamma_0,\beta)&=&\Delta_n(\psi_n,\gamma_0,\beta)+(\psi_0-\psi_n)^{\top}\partial_{\psi}\Delta_n({\psi}_n,\gamma_0,\beta) \nonumber \\
&&+\dfrac{1}{2}(\psi_0-\psi_n)^{\top}\partial_{\psi}^2 \Delta_n(\widetilde{\psi}_n,\gamma_0,\beta)(\psi_0-\psi_n). \notag
\end{eqnarray}
First, we show that, under $H_0$, as $n\rightarrow\infty,$
\begin{equation} \label{ps}
(\psi_0-\psi_n)^{\top}\partial_{\psi}^2 \Delta_n(\widetilde{\psi}_n,\gamma_0,\beta)(\psi_0-\psi_n)=o_P(1).
\end{equation}
For this, we observe that  
\begin{eqnarray*}
&& \left| (\psi_0-\psi_n)^{\top}\partial_{\psi}^2 \Delta_n(\widetilde{\psi}_n,\gamma_0,\beta)(\psi_0-\psi_n) \right|
\leq || \sqrt{n}(\psi_0-\psi_n)||\times \dfrac{1}{\sqrt{n}}||\partial_{\psi}^2 \Delta_n(\widetilde{\psi}_n,\gamma_0,\beta)||_M \times ||\psi_0-\psi_n||.
\end{eqnarray*}
From Proposition \ref{pr1}, we have that, under $H_0$, as $n\rightarrow\infty,$
$\sqrt{n} (\psi_0-\psi_n)$ converges in distribution to a normal distribution and $\psi_0-\psi_n$ tends to $\mathbf{0}$ in probability as $n\rightarrow\infty$. 
So (\ref{ps}) would be handled if 
we show that $||\partial_{\psi}^2\Delta_n(\widetilde{\psi}_n,\gamma_0,\beta)||_M/{\sqrt{n}}$ tends in probability to some positive random variable.
\smallskip

\noindent
To get this, we define for any $(\psi,\gamma,\beta)\in \mathbb R^{l+q} \times \mathbb R^{k+1} \times  \mathbb R^{k+1},$ 
$$D_n(\psi,\gamma,\beta)=\dfrac{1}{\sqrt{n}}\partial_{\psi}^2 \Delta_n(\psi,\gamma,\beta)=
 \begin{pmatrix} 
D_{n,1,1}(\psi,\gamma,\beta)&D_{n,1,2}(\psi,\gamma,\beta)\\ 
  D_{n,1,2}^{\top}(\psi,\gamma,\beta)& D_{n,2,2}(\psi,\gamma,\beta)
 \end{pmatrix},$$
%
 \begin{eqnarray*}
 D_{n,1,1}(\psi,\gamma,\beta)& := & \dfrac{1}{\sqrt{n}}\partial_{\rho}^2\Delta_n({\psi},\gamma,\beta)\\
 D_{n,1,2}(\psi,\gamma,\beta)& := & \dfrac{1}{\sqrt{n}}\partial_{\rho \theta}^2\Delta_n({\psi},\gamma,\beta)\\
  D_{n,2,2}(\psi,\gamma,\beta)& := & \dfrac{1}{\sqrt{n}}\partial_{\theta}^2\Delta_n({\psi},\gamma,\beta).
 \end{eqnarray*}
%
%
\noindent 
From  a lengthy but simple algebra, by ergodicity, it is easy to show that  
$||D_{n,1,1}(\widetilde{\psi}_n,\gamma_0,\beta)||_M,$ $||D_{n,2,2}(\widetilde{\psi}_n,\gamma_0,\beta)||_M$ and $||D_{n,1,2}(\widetilde{\psi}_n,\gamma_0,\beta)||_M$ converge in probability to some positive numbers.
It results from all these convergences that $||\partial_{\psi}^2\Delta_n(\widetilde{\psi}_n,\gamma_0,\beta)||_M/{\sqrt{n}}$ tends in probability to some positive number. Thus (\ref{ps}) is obtained.

\noindent
Now, as $n\rightarrow\infty,$ adding and subtracting appropriate terms, we can write 
\begin{eqnarray*}
\Delta_n(\psi_0,\gamma_0,\beta)
&=& \Delta_n(\psi_n,\gamma_0,\beta)+(\psi_0-\psi_{N(n)})^{\top}\partial_{\psi} \Delta_n({\psi}_n,\gamma_0,\beta)\\
&+&(\psi_{N(n)}-\psi_n)^{\top}\partial_{\psi} \Delta_n({\psi}_n,\gamma_0,\beta)+o_P(1).
\end{eqnarray*}
Observing that, as $n\rightarrow\infty,$ 
 $$\sqrt{n}(\psi_0-\psi_{N(n)})=\sqrt{N(n)}(\psi_0-\psi_{N(n)})\dfrac{\sqrt{n}}{\sqrt{N(n)}}=o_P(1),$$
it is easy to see that $\partial_{\psi} \Delta_n({\psi}_n,\gamma_0,\beta)/{\sqrt{n}}$ converges in probability to some random vector. So, we have proved that, as $n\rightarrow\infty,$ $$(\psi_0-\psi_{N(n)})^{\top}\partial_{\psi} \Delta_n({\psi}_n,\gamma_0,\beta)=\sqrt{n}(\psi_0-\psi_{N(n)})^{\top}\dfrac{1}{\sqrt{n}} \Delta_n({\psi}_n,\gamma_0,\beta)=o_P(1).$$
Consequently, as $n\rightarrow\infty,$
$$\Delta_n(\psi_0,\gamma_0,\beta)= \Delta_n(\psi_n,\gamma_0,\beta)+(\psi_{N(n)}-\psi_n)^{\top}\partial_{\psi} \Delta_n({\psi}_n,\gamma_0,\beta)+o_P(1).$$
For ending the proof of Proposition \ref{pr2}, we need the following lemma.
\begin{lemma} \label{lm1}
Let $\psi_n$ be a consistent and asymptotically normal estimator of $\psi_0$. For $\gamma_0,\beta\in\mathbb R^{k+1},$ 
$ \psi_{N(n)}$ is asymptotically in the tangent space  $\Gamma_n$ to the curve of $\Delta_n(\psi,\gamma_0,\beta)$ at $\psi_n$, defined as follows:
\begin{equation*}
\Gamma_n:=\left\lbrace x\in\mathbb R^{l+q},\Delta_n(x,\gamma_0,\beta)=\Delta_n(\psi_n,\gamma_0,\beta)+(x-\psi_n)\partial_{\psi}\Delta_n({\psi}_n,\gamma_0,\beta)\right\rbrace,
\end{equation*}
where $\lbrace N(n)\rbrace_{n\geq 1}$ stands for a subset $\lbrace 1,\ldots,n\rbrace$ such that $n/N(n),$ tends to 0 as $n$ tends to infinity.
\end{lemma}
\begin{proof}
writing a second-order Taylor expansion of $\Delta(\cdot, \gamma_0, \beta)$ around $\psi_n$, we have, for some $\widetilde{\psi}_n$ lying  between $\psi_{N(n)}$ and $\psi_n$,
\begin{eqnarray*}
\Delta_n(\psi_{N(n)},\gamma_0,\beta)&=&\Delta_n(\psi_n,\gamma_0,\beta)+(\psi_{N(n)}-\psi_n)^{\top}\partial_{\psi}\Delta_n(\psi_n,\gamma_0,\beta)\\
&+&\dfrac{1}{2}(\psi_{N(n)}-\psi_n)^{\top}\partial_{\psi}^2\Delta_n(\widetilde{\psi}_{N(n)},\gamma_0,\beta) (\psi_{N(n)}-\psi_n).
\end{eqnarray*}
To show that the sequence $\psi_{N(n)}$ is asymptotically in $\Gamma_n$, we just have to  show that $(\psi_{N(n)}-\psi_n)^{\top}\partial_{\psi}^2\Delta_n(\widetilde{\psi}_{N(n)},\gamma_0,\beta) (\psi_{N(n)}-\psi_n)=o_P(1).$ But since $n/N(n)\longrightarrow 0$ as $n\rightarrow\infty,$ we have 
\begin{eqnarray*}
\sqrt{n}(\psi_{N(n)}-\psi_n)
&=& \sqrt{N(n)}(\psi_{N(n)}-\psi_0)\dfrac{\sqrt{n}}{\sqrt{N(n)}}+\sqrt{n}(\psi_0-\psi_n)
\\&=& 
o_P(1)+\sqrt{n}(\psi_0-\psi_n).
\end{eqnarray*}
This implies that $\sqrt{n}(\psi_{N(n)}-\psi_n)$ converges in distribution to a Gaussian random vector.
So, to show that $(\psi_{N(n)}-\psi_n)^{\top}\partial_{\psi}^2\Delta_n(\widetilde{\psi}_{N(n)},\gamma_0,\beta) (\psi_{N(n)}-\psi_n)=o_P(1)$, we just show that the sequence $||\partial_{\psi}^2\Delta_n(\widetilde{\psi}_{N(n)},\gamma_0,\beta)||/\sqrt{n}$ converges in probability to some positive random variable. In this purpose, observing that 
$\widetilde{\psi}_{N(n)}-\psi_0=o_P(1),$
as we have shown previously that
$||\partial_{\psi}^2\Delta_n(\widetilde{\psi}_n,\gamma_0,\beta)||_M/\sqrt{n}$ converges in probability to some random variable, it follows that
$||\partial_{\psi}^2\Delta_n(\widetilde{\psi}_{N(n)},\gamma_0,\beta)||_M/\sqrt{n}$  converges in probability to some positive random variable.
Consequently,
$$(\psi_{N(n)}-\psi_n)^{\top}\partial_{\psi}^2\Lambda_n(\widetilde{\psi}_{N(n)},\gamma_0,\beta) (\psi_{N(n)}-\psi_n)=o_P(1).$$
This ends the proof of the lemma.
\end{proof}
%
%

\noindent
Now, by Lemma \ref{lm1} as $n\rightarrow\infty,$ we have 
  $$\Delta_n(\psi_{N(n)},\gamma_n)=\Delta_n(\psi_n,\gamma_0,\beta)+(\psi_{N(n)}-\psi_n)^{\top}\partial_{\psi} \Delta_n({\psi}_n,\gamma_0,\beta)+o_P(1).$$
Therefore, from the equality before the statement of  Lemma \ref{lm1}, we have that under $H_0$, as $n\rightarrow\infty,$ 
$$\Delta_n(\psi_0,\gamma_0,\beta)=\Delta_n(\psi_{N(n)},\gamma_0,\beta)+o_P(1).$$
This ends the proof the first part of Proposition \ref{pr2}.
\medskip

\noindent
For the second part, recall that for all $j=1, \ldots,{k+1}$,
$$ \mu_{j2}(\psi_0, \gamma_0)= I(f) \int_{\mathbb R^p} {1 \over V_{\theta_0}^2(x)} dF_j(x).$$
Now, adding and subtracting appropriate terms, we have 
\begin{eqnarray*}
\widehat \mu(\psi_n, \gamma_0)-\mu(\psi_0, \gamma_0)
&=&
\sum_{j=1}^{k+1} \beta_j^2\left \{ (\widehat \alpha_j-\alpha_j)[\widehat \mu_{j2}(\psi_n, \gamma_0) - \mu_{j2}(\psi_0, \gamma_0)] \right \}\\
&+& 
\sum_{j=1}^{k+1} \beta_j^2\left \{ \alpha_j[\widehat \mu_{j2}(\psi_n, \gamma_0) - \mu_{j2}(\psi_0, \gamma_0)] \right \}.
\end{eqnarray*}
Remembering that for all $j=1,\ldots k+1$,  $\widehat \alpha_j-\alpha_j$ tends to 0 as $n$ tends to infinity, 
by the continuity of the function $(x, \theta) \mapsto V_{\theta}(x)$ and the fact that it is bounded from the bottom by some positive number $\tau$, it follows from the Lebesgue's convergence theorem that each term in the above sum tends to 0. This handles the second part of Proposition \ref{pr2}.

\subsection{Proof of Theorem \ref{th3}}
\noindent We recall that 
 $$ \mathcal{T}_n(\psi_{N(n)},{\gamma}_{0},\beta)=  \dfrac{\Delta_n(\psi_{N(n)},\gamma_0,\beta)}{\widehat{\varpi}_n(\psi_{N(n)},\gamma_0,\beta)}.$$
From Proposition \ref{pr2}, under $H_0$, as $n\rightarrow\infty,$
$\Delta_n(\psi_0,\gamma_0,\beta)=\Delta_n(\psi_{N(n)},{\gamma}_{0},\beta)+o_P(1)$ 
and in probability, $\widehat{\varpi}_n(\psi_{N(n)},\gamma_0,\beta)\longrightarrow  \varpi(\psi_0,\gamma_0,\beta).$  
%
Hence the part $(i)$ follows from Theorem \ref{th1} and Proposition \ref{pr2}.

For the part $(ii)$, we first observe that by contiguity, the above asymptotics still hold under $H_1^{(n)}$. Next, from an adaptation of the proof of Corollary \ref{cr1}, we can see that under $H_1^{(n)},$ as $n\rightarrow\infty,$
$$\dfrac{\Delta_n(\psi_0,\gamma_0,\beta)}{\varpi(\psi_0,\gamma_0,\beta)}\overset{D}{\longrightarrow} \mathcal{N}(\varpi(\psi_0,\gamma_0,\beta),1).$$
Now, writing
\begin{eqnarray*}
\mathcal{T}_n(\psi_{N(n)},{\gamma}_{0},\beta)
&= &\dfrac{\Delta_n(\psi_{0},{\gamma}_{0},\beta)}{{\varpi}(\psi_{0},\gamma_0,\beta)}\times \dfrac{{\varpi}(\psi_{0},\gamma_0,\beta)}{\widehat{\varpi}_n(\psi_{N(n)},\gamma_0,\beta)}+\dfrac{1}{\widehat{\varpi}_n(\psi_{N(n)},\gamma_0,\beta)}\times o_P(1), 
\end{eqnarray*}
 it is easy to see that, under $H^{(n)}_{\beta},$ as $n\rightarrow \infty,$ $$\mathcal{T}_n(\psi_{N(n)},{\gamma}_{0},\beta)\overset{D}{\longrightarrow} \mathcal{N}(\varpi(\psi_0,\gamma_0,\beta),1).$$ 
Thus, we conclude that the local asymptotic power of the constructed test is preserved. Replacing 
 $\Delta_n(\psi_0,{\gamma}_{0},\beta)$ with its estimated version has no effect. Then, the statistics $\mathcal{T}_n(\psi_{N(n)},{\gamma}_{0},\beta)$ and $\mathcal{T}_n(\psi_0,\gamma_0,\beta)$ are locally asymptotically equivalent. Consequently, the last part of the theorem can be handled as that of Theorem \ref{th2}.
\subsection{Proof of Proposition \ref{pr3}}
\noindent For the proof of the proposition we need the following lemmas.
\begin{lemma} \label{lm2}
Assume $(A_1)$-$(A_7)$, $(B_0)$-$(B_4)$ hold. Then, for any sequence of consistent and asymptotically normal estimators $\lbrace\widehat{\gamma}_{0,n}\rbrace_{n\geq 1}$ of $\gamma_0$, under $H_0$, as $n\rightarrow\infty,$ we have, for any $\beta \in \mathbb R^{k+1}$, 
$$\Delta_n(\psi_0,{\gamma}_0,\beta)=\Delta_n(\psi_0,\widehat{\gamma}_{ 0,N(n)},\beta)+o_P(1),$$
where  $\lbrace N(n)\rbrace_{n\geq 1}$ stands for a subset $\lbrace 1,\ldots,n\rbrace$ such that  $n/N(n)$ tends to 0 as $n$ tends to infinity.
\end{lemma}
\begin{proof}
The proof is similar to that of Proposition \ref{pr2}.
\end{proof}
\begin{lemma} \label{lm3}
\noindent Let $\psi_n$ be a consistent estimator of $ \psi_0 $. Assume that $(A_1)$-$(A_7)$, $(B_0)$-$(B_4)$ hold. Then, for any sequence of consistent and asymptotically normal estimators $\lbrace\widehat{\gamma}_{0,n}\rbrace_{n\geq 1}$ of $\gamma_0$, under $H_0$, as $n\rightarrow\infty,$ we have, for any $\beta \in \mathbb R^{k+1}$,
$$\Delta_n(\psi_n,{\gamma}_0,\beta)=\Delta_n(\psi_n,\widehat{\gamma}_{0,{N(n) }},\beta)+o_P(1),$$
with $\lbrace N(n)\rbrace_{n\geq 1}$ standing for a subset $\lbrace 1,\ldots,n\rbrace$ such that  $n/N(n)$ tends to 0 as $n$ tends to infinity.
\end{lemma}
\begin{proof}
By a first-order Taylor expansion of $\Delta_n(\psi_n, \cdot,\beta)$ around $\widehat{\gamma}_{0,n},$ for some  $\dot{\gamma}_{0,n}$ lying between $\widehat{\gamma}_{0,n}$ and $\gamma_0,$ we have,
\begin{eqnarray*}
\Delta_n(\psi_n,\gamma_0,\beta)&= & \Delta_n(\psi_n,\widehat{\gamma}_{0,n},\beta)-(\widehat{\gamma}_{0,n}-{\gamma}_0)^{\top} \partial_{\gamma} \Delta_n (\psi_n,\widehat{\gamma}_{0,n},\beta)\\
&+& \dfrac{1}{2}(\widehat{\gamma}_{0,n}-{\gamma}_0)^{\top} \partial_{\gamma}^2 \Delta_n(\psi_n,\dot{\gamma}_{0,n},\beta) (\widehat{\gamma}_{0,n}-{\gamma}_0).
\end{eqnarray*}
Proceeding as in the proof of Proposition \ref{pr2}, we have that as $n\rightarrow\infty,$ 
$$({\gamma}_0-\widehat{\gamma}_{0,n})^{\top}  \partial_{\gamma}^2 \Delta_n(\psi_n,\widetilde{\gamma}_{0,n},\beta) ({\gamma}_0-\widehat{\gamma}_{0,n})=o_P(1).$$
Thus, as $n\rightarrow\infty,$ writing
\begin{eqnarray*}
\Delta_n(\psi_n,{\gamma}_0,\beta)&=& \Delta_n(\psi_n,\widehat{\gamma}_{0,n},\beta)+({\gamma}_0-\widehat{\gamma}_{0,N(n)})^{\top} \partial_{\gamma} \Delta_n (\psi_n,\widehat{\gamma}_{0,n},\beta)\\
&+&(\widehat{\gamma}_{0,N(n)}-\widehat{\gamma}_{0,n})^{\top}\partial_{\gamma} \Delta_n (\psi_n,\widehat{\gamma}_{0,n},\beta)+o_P(1)
\end{eqnarray*}
and observing that 
$\sqrt{n}(\gamma_0-\widehat{\gamma}_{0,N(n)})=o_P(1),$
it remains to show that $(1/\sqrt{n})\partial_{\gamma} \Delta_n (\psi_n,\widehat{\gamma}_{0,n},\beta)$ converges in probability to some random vector. In this purpose, for some $\ddot{\gamma}_{0,n}$ lying between $\gamma_0$ and $\widehat{\gamma}_{0,n}, $ we can write 
\begin{eqnarray}
\dfrac{1}{\sqrt{n}} \partial_{\gamma} \Delta_n (\psi_n,\widehat{\gamma}_{0,n},\beta) 
&=&\dfrac{1}{n}\displaystyle\sum_{t=1}^n{ \dfrac{\omega(t) \beta^{\top} \omega(t) }{V_{\theta_n}^2(Z_{t-1})}\phi'_f[ \varepsilon_t(\psi_n,{\gamma}_0)]   }
\notag \\&&
-(\widehat{\gamma}_{0,n}-\gamma_0)\dfrac{1}{n}\displaystyle\sum_{t=1}^n{ \dfrac{\omega(t)\beta^{\top}\omega(t)\omega^{\top}(t) }{V_{\theta_n}^3(Z_{t-1})}\phi''_f[ \varepsilon_t(\psi_n,\ddot{\gamma}_{0,n})]} \notag \\
&=& V_{1n}+V_{2n}
\end{eqnarray}
Under the assumptions $ (B_0) $ and $ (B_4) $ and the fact that $\widehat{\gamma}_{0,n}-\gamma_0$ converges in probability to 0, the ergodic theorem allows us to conclude that, as $n\rightarrow\infty,$ $V_{2n}$ tends in probability to 0. 
Now, adding and subtracting appropriate terms, and using a Taylor expansion, it is easy so see that 
\begin{eqnarray*}
V_{1n}&=&(\theta_n-\theta_0)\dfrac{1}{n}\displaystyle\sum_{t=1}^n{ \dfrac{\omega(t)\beta^{\top}  \omega(t) }{V_{\theta_n}^2(Z_{t-1})}\phi''_f[ \varepsilon_t(\widetilde{\psi}_n,{\gamma}_0)]\partial_{\theta}^{\top} \varepsilon_t (\widetilde{\psi}_n,{\gamma}_0)}\\
&&+(\rho_n-\rho_0)\dfrac{1}{n}\displaystyle\sum_{t=1}^n{ \dfrac{\omega(t)\beta^{\top}\omega(t)   }{V_{\theta_n}^2(Z_{t-1})}\phi''_f[ \varepsilon_t(\widetilde{\psi}_n,{\gamma}_0)]\partial_{\rho}^{\top} \varepsilon_t (\widetilde{\psi}_n,{\gamma}_0)}
-\dfrac{1}{n}\displaystyle\sum_{t=1}^n{ \dfrac{\omega(t) \beta^{\top} \omega(t) }{V_{\theta_n}^2(Z_{t-1})}\phi'_f[ \varepsilon_t(\psi_0,{\gamma}_0)]   }\\
&=& V_{11n}+V_{12n}+V_{13n}.
\end{eqnarray*}
In view of $(B_0)$ and the ergodic theorem, we can easily show that as $n$ tends to $\infty$, $V_{13n}$ converges in probability to some random vector.
\smallskip

\noindent
For the convergence of the term $V_{12n},$ we can write 
 $$V_{12n}=(\rho_n-\rho_0)\dfrac{1}{n}\displaystyle\sum_{t=1}^n{ \dfrac{\omega(t) \beta^{\top} \omega(t) }{V_{\theta_n}^2(Z_{t-1})}\phi''_f[ \varepsilon_t(\widetilde{\psi}_n,{\gamma}_0)]\dfrac{\partial_{\rho}^{\top} T_{\widetilde{\rho}_n}(Z_{t-1})}{V_{\widetilde{\theta}_n}(Z_{t-1})} },$$
from which we have 
$$||V_{12n}||\leq Cst ||\rho_n-\rho_0||\dfrac{1}{n}\displaystyle\sum_{t=1}^n{ {|\beta^{\top}\omega(t)| }\vartheta(Z_{t-1})}.$$
By the ergodic theorem and the fact that $\rho_n-\rho_0=o_P(1)$, it is easy to see that as $n\rightarrow\infty,$ the right-hand side of the above inequality tends in probability to 0. 
Therefore, as $n\rightarrow\infty,$ $V_{12n}$ tends in probability to 0.
\smallskip

\noindent
For the convergence of $V_{11n},$ we can write 
\begin{eqnarray} \label{v}
V_{11n}&=& -(\theta_n-\theta_0)\dfrac{1}{n}\displaystyle\sum_{t=1}^n \dfrac{\omega(t) \beta^{\top} \omega(t) }{V_{\theta_n}^2(Z_{t-1})}\phi''_f[ \varepsilon_t(\widetilde{\psi}_n,{\gamma}_0)]\partial_{\theta}^{\top}V_{\widetilde{\theta}_n}(Z_{t-1})\nonumber 
\times
 (\varepsilon_t( \widetilde{\psi}_n,{\gamma}_0)-\varepsilon_t( \psi_0,\gamma_0))
\nonumber \\
&&-(\theta_n-\theta_0)\dfrac{1}{n}\displaystyle\sum_{t=1}^n{ \dfrac{\omega(t)\beta^{\top}\omega(t)   }{V_{\theta_n}^2(Z_{t-1})}\phi''_f[ \varepsilon_t(\widetilde{\psi}_n,{\gamma}_0)]\partial_{\theta}^{\top}V_{\widetilde{\theta}_n}(Z_{t-1})\varepsilon_t( {\psi}_0,{\gamma}_0)}.
\end{eqnarray}
It is easy to see that the first term in the right-hand side of (\ref{v}) is bounded by
\begin{eqnarray*}
Cst ||\theta_0-\theta_n||^2\dfrac{1}{n}\displaystyle\sum_{t=1}^n{ {|\beta^{\top}\omega(t) | }\vartheta^2(Z_{t-1})|\varepsilon_t(\psi_0,\gamma_0)|}
+Cst ||\theta_n-\theta_0||\times ||\rho_0-\rho_n||\dfrac{1}{n}\displaystyle\sum_{t=1}^n{ {|\beta^{\top}\omega(t) | }\vartheta^2(Z_{t-1})}.
\end{eqnarray*}
Using again the ergodic theorem and the fact that $||\psi_n-\psi_0||$ converges in probability to 0, the right-hand of (\ref{v}) converges in probability to 0.
\smallskip

\noindent
In the same way, by the ergodic theorem, as $n\rightarrow\infty,$ the second term in the right-hand of (\ref{v}) converges almost surely to 0.
Then, as $n\rightarrow\infty,$ $V_{11n}$ tends in probability to 0. It results that $V_{1n}$ tends in probability to 0, as $n\rightarrow\infty.$ 
Consequently, as $n\rightarrow\infty,$  $$(\gamma_0-\widehat{\gamma}_{0,N(n)})^{\top} \partial_{\gamma} \Delta_n (\psi_n,\widehat{\gamma}_{0,n},\beta)=o_P(1),$$
$$\Delta_n(\psi_n,\gamma_0,\beta)=\Delta_n(\psi_n,\widehat{\gamma}_{0,n},\beta)+(\widehat{\gamma}_{0,N(n)}-\widehat{\gamma}_{0,n})^{\top}\partial_{\gamma}  \Delta_n (\psi_n,\widehat{\gamma}_{0,n},\beta)+o_P(1).$$
Now, for showing that $\widehat{\gamma}_{0,N(n)}$ is asymptotically in the tangent space to the curve of $\Delta_n (\psi_n, \gamma, \beta) $ at $ \widehat {\gamma}_{0, n}$, we proceed as in the proof of Lemma \ref{lm1} and obtain 
$$\Delta_n(\psi_n,\widehat{\gamma}_{0,N(n)},\beta)=\Lambda_n(\psi_n,\widehat{\gamma}_{0,n},\beta)+(\widehat{\gamma}_{0,N(n)}-\widehat{\gamma}_{0,n})^\top \partial_{\gamma}\Delta_n(\psi_n,\widehat{\gamma}_{0,n},\beta)+o_P(1).$$
Consequently, under $H_0,$ as $n\rightarrow\infty,$
$$\Delta_n(\psi_n,{\gamma}_0,\beta)=\Delta_n(\psi_n,\widehat{\gamma}_{0,N(n)},\beta)+o_P(1),$$
where $\lbrace N(n)\rbrace_{n\geq 1}$ stands for a subset $\lbrace 1,\ldots,n\rbrace$ such that $n/N(n)$ tends to 0 as $n$ tends to infinity.
\smallskip

\noindent Back to the proof of the proposition, from Lemma \ref{lm3}, for any estimator $\psi_n$ of $\psi_0$ satisfying $(B_4)$, for any sequence of consistent and asymptotically normal estimators $\lbrace\widehat{\gamma}_{0,N(n)}\rbrace_{n\geq 1}$ of $\gamma_0$, as $n\rightarrow\infty,$ 
$$\Delta_n(\psi_{N(n)},{\gamma}_0,\beta)=\Delta_n(\psi_{N(n)},\widehat{\gamma}_{0,{N(n) }},\beta)+o_P(1),$$
from which, subtracting and adding $\Delta_n(\psi_{N(n)},\gamma_0,\beta)$ we obtain 
\begin{eqnarray*}
\Delta_n(\psi_{N(n)},\widehat{\gamma}_{0,{N(n) }},\beta)-\Delta_n(\psi_0,{\gamma}_0,\beta)
&=&\Delta_n(\psi_{N(n)},\gamma_0,\beta)-\Delta_n(\psi_0,{\gamma}_0,\beta)+o_P(1).
\end{eqnarray*}
By Proposition \ref{pr2} and Lemma \ref{lm3}, we have 
$$\Delta_n(\psi_0,{\gamma}_0,\beta)=\Delta_n(\psi_{N(n)},\widehat{\gamma}_{0,{N(n) }},\beta)+o_P(1).$$
\end{proof}
%

\noindent
{\bf Acknowledgements} - 
We thank 
Professor Lihua XIONG and Professor Cong JIANG  for providing us with the Upper Hanjiang River floods data. 
\bibliographystyle{apalike}

\end{document}